\documentclass[a4paper,twoside]{article}      
\usepackage{amsmath,amssymb,amsfonts,amsthm,amscd} 
\usepackage{graphics}                 
\usepackage{color}                    
\usepackage{ mathrsfs }
\usepackage{indentfirst}
\usepackage{bbold}
\usepackage{enumerate}
\usepackage{url}         
\usepackage{colonequals} 
\usepackage{a4wide}

\newtheorem{theorem}{Theorem}[section]
\newtheorem{proposition}[theorem]{Proposition}
\newtheorem{corollary}[theorem]{Corollary}
\newtheorem{lemma}[theorem]{Lemma}
\theoremstyle{definition}
\newtheorem{remark}[theorem]{Remark}
\newtheorem{definition}[theorem]{Definition}

\def\div{\mathop{\mathrm{div}}\nolimits}
\def\!{\mathop{\mathrm{!}}}

\newcommand{\Keywords}[1]{\par\indent
{\small{\textbf{Keywords.} \/} #1}}
\newcommand{\subjclass}[1]{\par\indent{ \textbf{AMS subject classification.} #1}}

\DeclareMathOperator{\Hess}{Hess}

\DeclareMathOperator{\cov}{cov}
\DeclareMathOperator{\Ent}{Ent}

\def\idx{\mathrm{id_X}}
\def\R{\mathbf{ R}}
\def\T{\mathbb{T}}

\newlength{\boxwidth}
\title{The two-scale approach to hydrodynamic limits for
non-reversible dynamics}
\author{Manh Hong Duong\footnote{Mathematics Institute, University of Warwick, UK. m.h.duong@warwick.ac.uk}\hspace{2mm}  and Max Fathi\footnote{LPMA, University Paris 6, France. max.fathi@etu.upmc.fr}}
\date{\today}
\begin{document}
\maketitle

\begin{abstract}
In \cite{GOVW09}, a new method to study hydrodynamic limits, called the two-scale approach, was developed for reversible dynamics. In this work, we generalize this method to a family of non-reversible dynamics. As an application, we obtain quantitative rates of convergence to the hydrodynamic limit for a weakly asymmetric version of the Ginzburg-Landau model endowed with Kawasaki dynamics. These results also imply local Gibbs behaviour, following a method of \cite{Fat13}.
\end{abstract}
\Keywords Two-scale approach, hydrodynamic limits, non-reversible dynamics.
\subjclass 60G99, 35K99.

\section{Introduction}

In this work, we are interested in generalizing the results of \cite{GOVW09} on hydrodynamic limits to the case of weakly asymmetric interacting spin systems. We obtain quantitative rates of convergence to the hydrodynamic limit for such dynamics. Our main contribution is to show that we can control the effects of the antisymmetric component of the dynamic.

A typical result of convergence to the hydrodynamic limit consists in proving that, under a suitable time-space scaling and for nice initial conditions, a random system with a large number of particles behaves like a deterministic object, given as the solution of a partial differential equation. We refer to \cite{KL} for an overview of the field. 

In \cite{GOVW09}, a new method to study such problems, called the two-scale approach, was developed. It consists in establishing estimates in Wasserstein distance between the distribution of the system and a well-chosen deterministic macroscopic state, given as the solution of a differential equation. The main elements are a coarse-graining argument and a logarithmic Sobolev inequality. It was applied to reversible dynamics of the form 
$$dX_t = -A\nabla H(X_t)dt + \sqrt{2A}dW_t,$$
on some Euclidean space, where $A$ is a symmetric positive definite matrix, $H$ is the Hamiltonian and $W$ is a Wiener process. In the case where $A$ and $H$ correspond to the Ginzburg-Landau model endowed with Kawasaki dynamics, they obtained scaling limits of the form
$$\frac{\partial \rho}{\partial t} = \frac{\partial^2}{\partial \theta^2} \varphi'(\rho).$$

In this work, we add an extra term to the previous dynamic, and study
$$dX_t = -A\nabla H(X_t)dt - J\nabla H(X_t)dt + \sqrt{2A}dW_t,$$
where $J$ is an antisymmetric matrix. This extra term makes the dynamic non-reversible, but does not modify the invariant measure. For the Ginzburg-Landau model, when $J$ is a discrete differentiation, we obtain a scaling limit of the form
$$\frac{\partial \rho}{\partial t} = \frac{\partial^2}{\partial \theta^2} \varphi'(\rho) +  \frac{\partial}{\partial \theta} \varphi'(\rho).$$

Our method is restricted to the case where the square of the antisymmetric part $-J^2$ is controlled by $A$ (in the sense of symmetric matrices). This is because if the antisymmetric component becomes dominant in the scaling limit, we would expect the limiting PDE to be hyperbolic (rather than parabolic), estimates in Wasserstein distances would not be adapted, and a different metric would be required.

These estimates in Wasserstein distance also allow us to study local Gibbs behaviour (which is a stronger form of convergence) by using an interpolation inequality, following a method developed in \cite{Fat13}. Additionally, we obtain quantitative rates of convergence for the microscopic free energy to its scaling limit.

In general, hydrodynamic limit results for non-reversible dynamics are significantly harder to prove than in reversible situations. However, it turns out that the two-scale approach can be generalized under natural assumptions with fairly elementary arguments. The additional arguments for our main abstract result mostly rely on the Cauchy-Schwartz inequality and Gronwall-type estimates. We obtain easily-checkable and natural conditions for convergence to hold. Our method is illustrated with an application to a weakly asymmetric version of Kawasaki dynamics for a continuous spin system.

The plan of the paper is as follows: in Section 2, we present the framework and our main results. Section 3 contains the proofs of our results in the abstract setting. In Section 4, we give the proofs of convergence to the hydrodynamic limit for the Ginzburg-Landau model endowed with a weakly asymmetric version of Kawasaki dynamics.

\vspace{3mm}

{\Large{\textbf{Notations}}}

\begin{itemize}
\item $C$ denotes a positive constant, which may vary from line to line, or even within a line.

\item $\nabla$ is the gradient, $\Hess$ stands for Hessian, $| \cdot |$ is the norm and $\langle \cdot, \cdot \rangle$ is an inner product. If necessary, a subscript will indicate the space on which these are taken.

\item $P^t$ is the adjoint of the operator $P$. 

\item $\Phi$ is the function defined by $\Phi(x) := x\log x$ on $\R_+$.

\item $\Ent_{\mu}(f) = \int{f (\log f) \mu} - \left(\int{f \mu} \right) \log \left( \int{f \mu} \right)$ is the entropy of the positive function $f$ with respect to the measure $\mu$.

\item $Z$ is a constant enforcing unit mass for a probability measure.
\end{itemize}

\section{Framework and main results}

\subsection{Abstract setting}
Let $X,Y$ be two Euclidean (or affine) spaces with $X\subset \R^N, Y\subset \R^M$ for some integers $N$ and $M$.  We think of $X$ as the microscopic space and $Y$ as the macroscopic space. $N$ and $M$ can then be thought of as the size of the microscopic and macroscopic data respectively. Let $A$ and $J$ be respectively a positive definite symmetric and an anti-symmetric linear operator on $X$. Let $H\colon X\rightarrow \R$ be a given function. We consider the stochastic dynamics on $X$ that is given by the following stochastic differential equation (SDE)  
\begin{equation}
dX_t=-A\nabla H(X_t)\,dt-J\nabla H(X_t)\,dt+\sqrt{2A}\,dW_t,
\end{equation}
where $W_t$ is a Wiener process, and $\sqrt{A}$ is the square root of the matrix $A$. When $J \neq 0$, this is a non-reversible process, and the Fokker-Planck equation associated to this SDE is 
\begin{equation}
\label{eq: micro eqn}
\partial_t(f\mu)=\div[\mu (A+J)\nabla f],
\end{equation}
where $\mu$ is the invariant measure of the dynamics, which is
$$\mu(dx) := \frac{1}{Z}\exp(-H(x))dx$$
and $f(t,\cdot )$ is the density of the law of $X_t$ with respect to $\mu$. Note that the addition of $J$ does not change the invariant measure. As far as we know, every currently used method for proving hydrodynamic limit results relies on explicit knowledge of the invariant measure. 

In the application we have in mind, which we shall present in the next section, $A$ will be the discrete Laplacian, and $J$ the discrete derivation.

We now introduce an abstract framework for the notion of coarse-graining operator. Let $P\colon X\rightarrow Y$ be a linear operator such that
\begin{equation}
\label{eq: operator P}
NPP^t=\mathrm{id}_Y,
\end{equation}
where $P^t$ is the adjoint operator of $P$. We think of $y = Px$ as the macroscopic state associated to the microscopic state $x$. This operator induces a decomposition of the invariant measure into a macroscopic component and a fluctuation component.  Let $\overline{\mu}(dy)=P_\sharp \mu$ be the push-forward of $\mu$ under the operator $P$ and $\mu(dx|y)$ be the conditional measure of $\mu$ given $Px=y$, i.e., for each $y$, $\mu(dx|y)$ is a probability measure on $X$ and satisfies that for any test function $g$
\begin{equation}
\label{eq: disintegration}
\int_X g(x) d\mu(x)=\int_Y \Big(\int_{Px=y}g(x)\mu(dx|y)\Big)\overline{\mu}(dy).
\end{equation}

Applying the technique in~\cite{GOVW09}, we show that under certain conditions, the macroscopic profile $y=Px$, with law given by $\overline{f}(t,y)=\int_{Px=y}f(t,x)\mu(dx)$, is close to the solution of the following differential equation
\begin{equation}
\label{eq: macro eqn}
\frac{d\eta}{dt}=-(\overline{A}+\overline{J})\nabla \overline{H}(\eta(t)).
\end{equation}

In this equation, $\overline{A}$ is a symmetric, positive definite operator on $Y$ and $\overline{J}$ is another operator on $Y$, defined by
\begin{equation}
\label{eq: macro A}
\overline{A}^{-1}=PA^{-1}NP^t, \quad \overline{J}=\overline{A}PA^{-1}NJP^t,
\end{equation}
and $\overline{H}\colon Y\rightarrow \R$ is the macroscopic Hamiltonian that satisfies
\begin{equation}
\label{eq: macro H}
\overline{\mu}(dy)=\exp(-N\overline{H}(y))\,dy.
\end{equation}
$\overline{A}$ and $\overline{J}$ can be thought of as macroscopic versions of $A$ and $J$.

In order to state the assumptions, we need to recall the definition of the Logarithmic Sobolev inequality (LSI). A probability measure $\nu\in \mathcal{P}(X)$ is said to satisfy an $\mathrm{LSI}$ with constant $\rho>0$ (abbreviation $\mathrm{LSI}(\rho)$) if, for any locally Lipschitz, nonnegative function $f\in L^1(\nu)$,  
\[
\int \Phi(f)\,d\nu-\Phi\left(\int f \,d\nu\right)\leq \frac{1}{2\rho}\int \frac{|\nabla f|^2}{f}\,d\nu.
\]
\textbf{Assumptions:} Throughout the paper, we assume that
\begin{enumerate}[(i)]
\item $\kappa\colonequals\max_{x\in X}\{\langle\Hess H(x)\cdot u,v\rangle, u\in \mathrm{Ran}(NP^tP),v\in\mathrm{Ran}(\idx - NP^tP),|u|=|v|=1\}<\infty$;
\item There is $\rho>0$ such that $\mu(dx|y)$ satisfies $\mathrm{LSI}(\rho)$ for all $y$;
\item There exist $\lambda,\Lambda>0$ such that $\lambda\mathrm{Id}\leq \mathrm{Hess}\overline{H}\leq \Lambda\mathrm{Id}$;\label{asumpt: bounded of Hess macro H}
\item There is $\alpha>0$ such that $\int_X|x|^2f\mu(dx)\leq \alpha N$;
\item There is $\beta>0$ such that $\inf_{y\in Y} \bar{H}(y)\geq -\beta$;
\item There is $\gamma>0$ such that for all $x\in X$,
\[
|(\idx-NP^tP)x|^2\leq \gamma M^{-2}\langle x,Ax\rangle_X;
\]
\item There are constants $C_1$ and $C_2$ such that the initial datum satisfy
\[
\int \Phi(f(0,x))\mu(dx)\leq C_1N\quad \text{and}\quad \bar{H}(\eta_0)\leq C_2;
\]
\item There is a $\tau>0$ such that $A\geq \tau Id$;\label{asumpt: bound below of A}
\item $-J^2\leq cA$\label{asumpt: weak anti-symmetric};
\item $J$ and $A$ commute.
\end{enumerate}

Under these assumptions, we have the following bound on the (scaled) Wasserstein distance between $f\mu$ and $\delta_{NP^t\eta}$ as well as the time-integrated Wasserstein distance between $\overline{f}\overline{\mu}$ and $\delta_{\eta}$.

\begin{theorem}
\label{theo: abstract theorem}
Let $\mu(dx)=Z^{-1}\exp(-H(x))\,dx$ be a probability measure on $X$, and let $P\colon X\rightarrow Y$ satisfy~\eqref{eq: operator P}. Let $A\colon X\rightarrow X$ be a symmetric, positive definite operator, and $f(t,x)$ and $\eta(t)$ be the solutions of~\eqref{eq: micro eqn} and~\eqref{eq: macro eqn}, with initial data $f(t,\cdot)$ and $\eta_0$ respectively. Suppose that the assumptions above hold. Define
\begin{equation}
\label{def: theta}
\Theta(t)\colonequals \frac{1}{2N}\int_X(x-NP^t\eta(t))\cdot A^{-1}(x-NP^t\eta(t))f(t,x)\mu(dx).
\end{equation}
Then for any $T>0$, we have 
\[
\max\Big\{\sup_{0\leq t\leq T}\Theta(t),\frac{\lambda}{8}\int_0^T\left(\int_Y|y-\eta(t)|_Y^2\overline{f}(t,y)\overline{\mu}(dy)\right)\,dt\Big\}\leq e^{\frac{8c\Lambda^2}{\lambda}T}\Big[\Theta(0)+E(T,M,N)\Big],
\]
where $E(T,M,N)\rightarrow 0$ as $N\uparrow \infty, M\uparrow \infty, \frac{N}{M}\uparrow \infty$. More precisely,
\begin{multline*}
E(T,M,N)=T\left(\frac{M}{N}\right)+\frac{4c\gamma \Lambda^2T}{\lambda}\left(\alpha+\frac{2C_1}{\hat{\rho}}\right)\frac{1}{M}+C_1\left(\frac{\gamma \kappa^2}{2\lambda\rho^2}+\frac{2c\gamma\kappa^2}{\tau\lambda\rho^2}+\frac{4\gamma c}{\lambda \tau}\right)\frac{1}{M^2}
\\\qquad+\sqrt{2T\gamma}\left(\alpha+\frac{2C_1}{\hat{\rho}}\right)^\frac{1}{2}\left\{\left(1+\sqrt{\frac{c}{\tau}}
+\frac{\sqrt{2c\gamma}}{M}\right)\sqrt{C_1}\right.
\\ \left.\hspace{15mm}+\sqrt{2}\left(1+\sqrt{\frac{c}{\tau}}\right)(H(\eta_0)-H(\eta_T))+CT(1+e^{CT}\overline{H}(\eta_0))^\frac{1}{2}\right\}\frac{1}{M},
\end{multline*}
where 
$$
\hat{\rho}:=\frac{1}{2}\left(\rho+\lambda+\frac{\kappa^2}{\rho}-\sqrt{\left(\rho+\lambda+\frac{\kappa^2}{\rho}\right)^2-4\rho\lambda}\right).
$$
\end{theorem}

\begin{remark}[Remarks on the assumptions] Assumptions (i) to (viii) are collected from~\cite{GOVW09} and~\cite{Fat13}, and where already used in the symmetric case. Assumption (ix) means that the asymmetric effect is controlled by the symmetric one.  Its main use is to rule out situations where the scaling limit is a hyperbolic equation (this would be the case for a continuous analog of the fully asymmetric exclusion process), which the two-scale approach does not seem to handle. Assumption (x) is natural if we think of $J$ and $A$ as finite approximations of first and second derivatives operators, which is the application we have in mind. It could be replaced by an appropriate bound on the symmetric part of $PA^{-1}JNP^t$ (which is the macroscopic component of the commutator between $A^{-1}$ and $J$), and an additional bound of the form $|\mathrm{Tr}(PJA^{-1}NP^t)| \leq CM$. But since our proof is already fairly technical, and we do not have an application in mind that would warrant the greater generality, we decided to just assume that $A$ and $J$ commute, and simplify the proof. All these assumptions will be used in Lemma 3.4 to estimate the time derivative of $\Theta(t)$. In particular, (ii) and (vi) are used to handle the covariance and fluctuations terms respectively. Assumption (iii) is used to control the macroscopic terms, and implies a LSI for the coarse-grained measure $\bar{\mu}$. 
\end{remark}

The hydrodynamic limit is obtained as a consequence.

\begin{corollary} \label{cor: abstract theorem}
Consider a sequence $\{X_{\ell}, Y_{\ell}, P_{\ell}, A_{\ell}, J_{\ell}, \mu_{\ell}, f_{0,\ell}, \eta_{0,\ell}\}_{\ell}$ satisfying the assumptions (i) to (x) with uniform constants $\kappa, \rho, \lambda, \Lambda, \alpha, \beta, \gamma, C_1, C_2$ and $c$.
Suppose that 
$$N_{\ell} \underset{\ell \uparrow \infty}{\longrightarrow} \infty; \hspace{3mm} M_{\ell} \underset{\ell \uparrow \infty}{\longrightarrow} \infty; \hspace{3mm} \frac{M_{\ell}}{N_{\ell}} \underset{\ell \uparrow \infty}{\longrightarrow} 0.$$
Further assume that 
$$\underset{\ell \uparrow \infty}{\lim} \frac{1}{N_{\ell}}\int{(x - N_{\ell}P^t\eta_{0,\ell})\cdot A_{\ell}^{-1}(x - N_{\ell}P^t\eta_{0,\ell}) f_{0,\ell}(x)\mu_{\ell}(dx)} = 0.$$
Then, for any $T > 0$: 

\noindent (a) The microscopic  variables are close to the solution of~\eqref{eq: macro eqn} in the penalized norm induced by $A_{\ell}^{-1}$, uniformly in $t \in [0,T]$: 
$$\underset{\ell \uparrow \infty}{\lim} \hspace{1mm} \underset{0 \leq t \leq T}{\sup} \hspace{1mm} \frac{1}{N_{\ell}}\int{(x - N_{\ell}P^t\eta_{\ell})\cdot A_{\ell}^{-1}(x - N_{\ell}P^t\eta_{\ell}) f_{\ell}(t,x)\mu_{\ell}(dx)} = 0;$$

\noindent (b) The macroscopic variables are  close to the solution of~\eqref{eq: macro eqn} in the strong $L^2(Y)$ norm, in a time-integrated sense: 
$$\underset{\ell \uparrow \infty}{\lim} \hspace{1mm} \int_0^T{\int{|y - \eta_{\ell}|_Y^2\bar{f}(t,y)\bar{\mu}(dy)}dt} = 0.$$
\end{corollary}

Another topic of interest is whether the data behaves like a local Gibbs state. 

\begin{definition}
The local Gibbs state with macroscopic profile $\eta \in Y$ is the probability measure on $X$ whose density with respect to $\mu$ is given by
$$G(x)\mu(dx) := \frac{1}{Z}\exp(NP^t\nabla \bar{H}(\eta) \cdot x)\mu(dx).$$
\end{definition}

Such a probability measure is close (in Wasserstein distance) to the associated macroscopic profile $\eta$.

In \cite{Yau}, Yau showed that, if the initial data is close (in the sense of relative entropy) to a local Gibbs state, then this also holds at any positive time, for a time-dependent local Gibbs state. Since closeness in relative entropy is stronger (in the current setting) than closeness in Wasserstein distance (this is a consequence of Talagrand's inequality, which is implied by the LSI, see for example \cite{OV}), the kind of results obtained with Yau's method are stronger than those of the previous Corollary, but require a stronger assumption on the initial data.

In \cite{Kos}, it was shown that convergence in relative entropy actually holds at positive times, even if the initial data only converges in a weaker sense. In \cite{Fat13}, the second author obtained a new proof of this fact in the reversible setting, using the two-scale approach. This method also yields quantitative rates of convergence in relative entropy. Now that we have generalized the two-scale approach to the non-reversible setting, the extension of the results of \cite{Fat13} follows.

\begin{theorem}
\label{theo: abstract entropy}
 Let $G(t,x)$ be the time-dependent local Gibbs state associated to the solution $\eta$ of~\eqref{eq: macro eqn}. Under our assumptions, the following holds : 

(a) The relative entropy with respect to the local Gibbs state is controlled as follows:  
\begin{equation}\int_0^T{\frac{1}{N}\int{\Phi\left(\frac{f(t,x)}{G(t,x)}\right)G(t,x) \mu(dx)}dt} = \text{\Large O} \left(\sqrt{\Theta(0) + \frac{M}{N} + \frac{1}{M}} \right),
\end{equation}
where the actual constants in the bound (which can be made explicit) depend on $T$, $\lambda$, $\Lambda$, $\alpha$, $\gamma$, $\rho$, $\kappa$, $\tau$, $c$, $C_1$ and $C_2$, but not on $M$ and $N$;

(b) The difference between the microscopic free energy and the free energy associated with the macroscopic profile $\eta$ is bounded as follows:
\begin{align}
\int_0^T&{\left|\frac{1}{N}\int{\Phi(f(t,x))\mu(dx)} - \bar{H}(\eta(t))\right|dt} \notag \\
&= \text{\Large O} \left(\sqrt{\Theta(0) + \frac{M}{N} + \frac{1}{M}} \right) \notag \\
&+ \text{\Large O} \left(\frac{M}{N}\right) \times \max\left(\left|\log\left(\frac{\Gamma(Y,|\cdot|_Y)^{2/(M-1)}}{\Lambda N}\right)\right|, \left|\log\left(\frac{\Gamma(Y,|\cdot|_Y)^{2/(M-1)}}{\lambda N}\right)\right|\right),
\end{align}
where $\Gamma(Y,|\cdot|_Y)= \int{\exp(-|y|_Y^2/2)dy}$ is the Gaussian integral on the (affine) space $Y$ with respect to the norm $|\cdot |_Y$.
\end{theorem}

\subsection{Application to spin systems}

We now give an application of Theorem \ref{theo: abstract theorem} to a system of interacting continuous spins. We consider collections of $N$ spins, in the space
$$X_{N,m} := \left\{x \in \R^N; \hspace{3mm} \frac{1}{N}\underset{i = 1}{\stackrel{N}{\sum}} \hspace{1mm} x_i = m \right\},$$
which we endow with the usual $\ell^2$ scalar product. The constraint $N^{-1}\sum x_i = m$ corresponds to a constraint of fixed mean spin, that will be preserved by the dynamics.

The application we have in mind is when the matrices $A$ and $J$ are given by the N-dimensional matrices 
\begin{equation} 
\label{matrix: A,J}
A=N^2\begin{pmatrix}
2&-1&&&-1\\
-1&2&-1&&\\
&\ddots&\ddots&\ddots&\\
&&-1&2&-1\\
-1&&&-1&2
\end{pmatrix},\quad J=\frac{N}{2}\begin{pmatrix}
0&1&&&-1\\
-1&0&1&&\\
&\ddots&\ddots&\ddots&\\
&&-1&0&1\\
1&&&-1&0
\end{pmatrix}.
\end{equation}
As in \cite{GOVW09}, let 
\begin{equation}
H(x) := \underset{i = 1}{\stackrel{N}{\sum}} \hspace{1mm} \psi(x_i),
\end{equation}
where $\psi : \R \longrightarrow \R$ satisfies the following assumptions: 

\begin{equation}
\label{assumption_potential}
\psi(x) = \frac{1}{2}x^2 + \delta\psi(x); \hspace{3mm} ||\delta\psi||_{C^2} < \infty.
\end{equation}
This assumption will ensure that (iii) holds.

We consider the dynamic where $A$ and $J$ are given by~\eqref{matrix: A,J} respectively. This corresponds to the system of $N$ stochastic differential equations

$$dX_i(t) = -N^2(2\psi(X_i) - \psi(X_{i+1}) - \psi(X_{i-1}))dt - \frac{N}{2}(\psi(X_{i+1}) - \psi(X_{i-1})dt + N\sqrt{2}(dB^{i+1}_t - dB^i_t).$$
The index $i$ goes from $1$ to $N$, and we impose periodic boundary conditions, that is $N+1 \equiv 1$. 

This is the dynamic studied in \cite{GPV} and \cite{GOVW09}, to which we have added a weak asymmetric perturbation. This model is to the symmetric dynamic what the weakly asymmetric exclusion process is to the simple symmetric exclusion process, i.e., we have added an extra asymmetric term which has a scaling of lower order in $N$.

Following \cite{GOVW09}, the macroscopic space is 
$$Y_{M,m} := \left\{ y \in \R^M; \hspace{3mm} \frac{1}{N}\underset{i = 1}{\stackrel{M}{\sum}} \hspace{1mm} y_i = m \right\},$$
which we endow with the $L^2$ scalar product
$$\langle y, \tilde{y} \rangle_Y := \frac{1}{M}\underset{i = 1}{\stackrel{M}{\sum}} \hspace{1mm} y_i \tilde{y}_i.$$
The coarse graining operator $P$ is defined as 
$$(Px)_i := \frac{1}{K}\underset{j = (i-1)K +1}{\stackrel{iK}{\sum}} \hspace{1mm} x_i,$$
where $K$ is an integer such that $N = KM$. We can think of this coarse-graining operator as taking local averages of the microscopic profile over boxes of size $K$. This operator does satisfy the relation $PNP^t = \text{id}_Y$.

The coarse-grained Hamiltonian is given by
\begin{equation}
\bar{H}(y) = \frac{1}{M}\underset{i = 1}{\stackrel{M}{\sum}} \hspace{1mm} \psi_K(y_i) + \frac{1}{N}\log \bar{Z},
\end{equation}
where
\begin{equation} \label{def: psi K}
\psi_K(m) := -\frac{1}{K}\log \int_{X_{K,m}}{\exp \left(-\underset{j = 1}{\stackrel{K}{\sum}} \hspace{1mm} \psi(x_i) \right)dx}.
\end{equation}
A classical result of large deviations theory (see for example \cite{GOVW09} for a proof) states that $\psi_K$ is close, in a local sense, to the Cramer transform of $\psi$, defined as
\begin{equation}
\label{eq: varphi}
\varphi(m)=\sup_{\sigma\in \R}\left\{\sigma m-\log{\int_\R\exp{(\sigma x-\psi(x))}}dx\right\}.
\end{equation}
More precisely, we have
$$||\psi_K - \varphi ||_{C^2} \underset{K \rightarrow \infty}{\longrightarrow} \hspace{1mm} 0.$$
As a consequence, since $\varphi$ is uniformly convex, and since $\psi_K''$ uniformly converges to  $\varphi''$, $\psi_K$ is uniformly convex as soon as $K$ is large enough. This shall allow us to apply the previous abstract theorem.

Without loss of generality, we shall assume in the sequel that $m = 0$, since it does not play a role in our estimates.

To study the scaling limit, we need to embed our spaces $X_{N,m}$ into a single functional space. To a microscopic profile $x \in X_{N,0}$, we associate the  step function on the torus $\bar{x}$, defined by

\begin{equation} \label{def: bar x}
\bar{x}(\theta) := x_i \hspace{3mm} \forall \theta \in \left[\frac{i-1}{N}, \frac{i}{N}\right).
\end{equation}

We endow the space $L^2(\T)$ with the $H^{-1}$ norm, defined by
$$||w||_{H^{-1}}^2 = \int{g^2d\theta}, \hspace{2mm} g' = w, \hspace{2mm} \int{g\hspace{1mm} d\theta} = 0.$$
The closure of the spaces $X_{N,0}$ for this norm is the usual $H^{-1}$ space of functions of average $0$, which is the dual of the Sobolev space $H^1$ for the $L^2$ norm.

We can now state the hydrodynamic limit result we obtain for this model : 

\begin{theorem}
\label{theo: concrete hydro}
Let $A_{\ell}$ and $J_{\ell}$ be given by~\eqref{matrix: A,J}. Assume that $\psi$ satisfies (\ref{assumption_potential}). Let $f(t,x)$ be a time-dependent probability density on $(X_{N,0}, \mu_{N,0})$ solving~\eqref{eq: micro eqn}, with $f(0,\cdot) = f_0$ such that 
$$\int{f_0 \log f_0 d\mu_{N,0}} \leq CN,$$
for some $C>0$ and 
$$\underset{N \uparrow \infty}{\lim} \hspace{1mm} \int{||\bar{x} - \zeta_0||_{H^{-1}}^2f_0(x)\mu_{N,0}(dx)} = 0,$$
for some initial macroscopic profile $\zeta_0 \in L^2(\T)$. Then, for any $T > 0$, we have
$$\underset{N \uparrow \infty}{\lim} \hspace{1mm} \underset{0\leq t \leq T}{\sup} \hspace{1mm}  \int{||\bar{x} - \zeta(t,\cdot)||_{H^{-1}}^2f(t,x)\mu_{N,0}(dx)} = 0,$$
where $\zeta$ is the unique solution of 
\begin{equation}
\label{eq: limiting eqn}
\begin{cases}
\frac{\partial \zeta}{\partial t}=\frac{\partial^2}{\partial \theta^2}\varphi'(\zeta)+ \frac{\partial}{\partial \theta} \varphi ' (\zeta),\\
\zeta(0,\cdot)=\zeta_0,
\end{cases}
\end{equation}
where $\varphi$ is the Cram\'{e}r transform of $\psi$, defined as (\ref{eq: varphi}).
\end{theorem}

We can also use the method of~\cite{Fat13} to study local Gibbs behaviour, and convergence of the relative entropy. 

\begin{theorem}
Under the same assumptions as in Theorem \ref{theo: concrete hydro}, the following holds : 
\begin{equation} \label{ent1}
\int_0^T{\int_{X_N}{\Phi\left(\frac{f_N(t,x)}{G_N(t,x)}\right)G_N(t,x) \mu_N(dx)}dt} \longrightarrow 0,
\end{equation}
where $G_N(t,\cdot)$ is the local Gibbs state given by $\eta_N(t)$. 
As a consequence, we have convergence of the microscopic entropy to the hydrodynamic entropy, in a time-integrated sense : 
\begin{equation} \label{ent2}
\int_0^T{\left|\frac{1}{N}\int{\Phi(f_N(t,x))\mu_N(dx)} - \left(\int_{\mathbb{T}}{\varphi(\zeta(\theta,t))d\theta} - \varphi\left(\int_{\mathbb{T}}{\zeta(t,\theta) d\theta}\right)\right)\right|dt} \underset{N \rightarrow \infty}{\rightarrow} 0.
\end{equation}

Moreover, convergence of $\frac{1}{N}\int{\Phi(f_N(t,x))\mu_N(dx)}$ to $\int_{\mathbb{T}}{\varphi(\zeta(\theta,t))d\theta} - \varphi\left(\int_{\mathbb{T}}{\zeta(t,\theta) d\theta}\right)$ holds uniformly on any time-interval  $[\epsilon, T]$, for any $0 < \epsilon < T$.
\end{theorem}

Since deducing this result from Theorem \ref{theo: abstract entropy}  is nearly the same as in \cite{Fat13}, we omit the proof. The only significant difference is proving that the solution $\zeta$ of the hydrodynamic equation \ref{eq: limiting eqn} is smooth on $[\epsilon, T]$, which is a well-known result in parabolic PDE theory (see for example \cite{LSU}). Alternatively, it can be proven by a straightforward adaptation of the proof of Proposition 3.22 in \cite{Fat13}.

\section{Proof of the abstract results}
In this section, we prove Theorem~\ref{theo: abstract theorem} and provide a sketch of proof of Theorem~\ref{theo: abstract entropy}.
\subsection{Proof of Theorem \ref{theo: abstract theorem}}
Following the approach of \cite{GOVW09}, we prove Theorem~\ref{theo: abstract theorem} in three steps : first we differentiate with respect to time the  (scaled) Wasserstein distance between $f(t)\mu$ and the macroscopic profile $\delta_{NP^t\eta(t)}$, and we split the expression into macroscopic components and fluctuations around the macroscopic profile. Then we derive an upper bound for the quantity we obtain, using assumption (iii) to control the macroscopic contribution and assumption (vi) to control fluctuations. Finally, we integrate in time and apply Gronwall's Lemma to obtain the result.

\begin{lemma}
\label{lema: time-derivative of Theta}
Let $\Theta$ be defined as in (\ref{def: theta}). We compute its time-derivative: 
\begin{align}
&\frac{d}{dt}\frac{1}{2N}\int_X(x-NP^t\eta(t))\cdot A^{-1}(x-NP^t\eta(t))f(t,x)\mu(dx)\nonumber
\\&\quad=\frac{M}{N}-\int_Y(y-\eta)\cdot(\nabla_Y\overline{H}(y)-\nabla_Y\overline{H}(\eta))\overline{f}(t,y)\overline{\mu}(dy)\nonumber
\\&\quad\quad-\int_YPJA^{-1}NP^t(y-\eta)\cdot(\nabla_Y\overline{H}(y)-\nabla_Y\overline{H}(\eta))\overline{f}(t,y)\overline{\mu}(dy)\nonumber
\\&\quad\quad-\int_Y(y-\eta)\cdot P\mathrm{cov}_{\mu(dx|y)}(f,\nabla H)\overline{\mu}(dy)\nonumber
\\&\quad\quad -\frac{1}{N}\int_X (\idx-NP^tP)x\cdot\nabla f(t,x)\mu(dx)\nonumber
\\&\quad\quad+\int \overline{A}\nabla_Y \overline{H}(\eta)\cdot PA^{-1}(\idx-NP^tP)xf\mu(dx)\notag
\\&\quad\quad+\int_YPJA^{-1}NP^t(y-\eta)\cdot P\mathrm{cov}_{\mu(dx|y)}(f,\nabla H)\overline{\mu}(dy) \nonumber
\\&\quad\quad+\int_Y PJA^{-1}(\idx-NP^tP)x\cdot P\nabla f(t,x)\mu(dx)\nonumber
\\&\quad\quad+\frac{1}{N}\int_X (\idx-NP^tP)JA^{-1}(x-NP^t\eta)\cdot \nabla f(t,x)\mu(dx)\nonumber
\\&\quad\quad+\int_X PA^{-1}(\idx-NP^tP)x\cdot \overline{J}\nabla_Y\overline{H}(\eta)f(t,x)\mu(dx)\label{eq: d/dt Theta}.
\end{align}
\end{lemma}
\begin{proof}
We have
\begin{align}
&\frac{d}{dt}\frac{1}{2N}\int_X(x-NP^t\eta(t))\cdot A^{-1}(x-NP^t\eta(t))f(t,x)\mu(dx)\nonumber
\\&\quad\stackrel{\eqref{eq: micro eqn}}{=}-\frac{1}{N}\int_X A^{-1}(x-NP^t\eta)\cdot (A+J)\nabla f\mu(dx)-\int P^t\frac{d\eta}{dt}\cdot A^{-1}(x-NP^t\eta)f\mu(dx)\nonumber
\\&\quad\stackrel{\eqref{eq: macro eqn}}{=}-\frac{1}{N}\int_X A^{-1}(x-NP^t\eta)\cdot A\nabla f\mu(dx)+\int \overline{A}\nabla_Y\overline{H}(\eta)\cdot PA^{-1}(x-NP^t\eta)f\mu(dx)\nonumber
\\&\quad\quad -\frac{1}{N}\int A^{-1}(x-NP^t\eta)\cdot J\nabla f\mu(dx)+\int A^{-1}(x-NP^t\eta)\cdot P^t\overline{J}\nabla \overline{H}(\eta)f\mu(dx)\nonumber
\\&=\quad (I)+(II)+(III)+(IV).\label{eq: d/dt theta 1}
\end{align}
We now use the decomposition $x=NP^tPx+(\idx-NP^tP)x$ to transform each term on the right hand side of~\eqref{eq: d/dt theta 1}. We need the following definition of the $\mu$-covariance of two functions $f,g\in L^2(\mu)$
\begin{equation}
\label{eq:covariance}
\cov_\mu(f,g)=\int fg\,d\mu-\left(\int f\,d\mu\right)\left(\int g\,d\mu\right).
\end{equation}

The first two terms, $(I)$ and $(II)$, are already done in~\cite{GOVW09}. We repeat the arguments here for the sake of completeness.
\begin{align}
&(I)=-\frac{1}{N}\int_X (x-NP^t\eta)\cdot \nabla f\mu(dx)\nonumber
\\&\quad=-\int_X P^t(Px-\eta)\cdot \nabla f\mu(dx)-\frac{1}{N}\int (\idx-NP^tP)x\cdot\nabla f\mu(dx).\label{eq: the first term of (I)}
\end{align}
We now transform the first term in~\eqref{eq: the first term of (I)} using~\eqref{eq: disintegration} and Lemma 21 in~\cite{GOVW09}. 
\begin{align*}
&-\int_X P^t(Px-\eta)\cdot \nabla f\mu(dx)=-\int (Px-\eta)\cdot P\nabla f\mu(dx) \nonumber
\\&\qquad\overset{\eqref{eq: disintegration}}{=} -\int_Y (y-\eta)\cdot P\int_{Px=y}\nabla f\mu(dx|y)\overline{\mu}(dy)\nonumber
\\&\qquad\overset{\cite[(36)]{GOVW09}}{=}-\frac{1}{N}\int(y-\eta)\cdot \nabla_Y\overline{f}\overline{\mu}(dy)-\int (y-\eta)\cdot P\mathrm{cov}_{\mu(dx|y)} (f,\nabla H)\overline{\mu}(dy)\nonumber
\\&\qquad\overset{\eqref{eq: macro H}}{=}\frac{1}{N}\int \nabla_Y\cdot y\overline{f}\overline{\mu}(dy)-\int(y-\eta)\cdot \nabla_Y\overline{H}(y)\overline{f}\overline{\mu}(dy)-\int (y-\eta)\cdot P\mathrm{cov}_{\mu(dx|y)} (f,\nabla H)\overline{\mu}(dy)\nonumber
\\&\qquad= \frac{\dim Y}{N}-\int(y-\eta)\cdot \nabla_Y\overline{H}(y)\overline{f}\overline{\mu}(dy)-\int (y-\eta)\cdot P\mathrm{cov}_{\mu(dx|y)} (f,\nabla H)\overline{\mu}(dy).
\end{align*}
We obtain
\begin{equation}
\label{eq: I}
(I)=\frac{\dim Y}{N}-\int(y-\eta)\cdot \nabla_Y\overline{H}(y)\overline{f}\overline{\mu}dy-\int (y-\eta)\cdot P\mathrm{cov}_{\mu(dx|y)} (f,\nabla H)\overline{\mu}dy-\frac{1}{N}\int (\idx-NP^tP)x\cdot\nabla f\mu(dx).
\end{equation}
Now we proceed with $(II)$. 
\begin{align}
(II)&=\int \overline{A}\nabla_Y \overline{H}(\eta)\cdot PA^{-1}NP^t(Px-\eta)f\mu(dx)+
\int \overline{A}\nabla_Y \overline{H}(\eta)\cdot PA^{-1}(\idx-NP^tP)xf\mu(dx)\nonumber
\\&\overset{\eqref{eq: macro A}}{=}\int\nabla_Y \overline{H}(\eta)\cdot (Px-\eta)f\mu(dx)+
\int \overline{A}\nabla_Y \overline{H}(\eta)\cdot PA^{-1}(\idx-NP^tP)xf\mu(dx)\nonumber
\\&\overset{\eqref{eq: disintegration}}{=}\int_Y (y-\eta)\cdot\nabla_Y\overline{H}(\eta)\overline{f}\overline{\mu}(dy) +\int \overline{A}\nabla_Y \overline{H}(\eta)\cdot PA^{-1}(\idx-NP^tP)xf\mu(dx).\label{eq: II}
\end{align}
Next, we continue with $(III)$.
\begin{align*}
(III)&=\frac{1}{N}\int JA^{-1}(x-NP^t\eta)\cdot \nabla f\mu(dx)
\\&=\frac{1}{N}\int PJA^{-1}(x-NP^t\eta)\cdot NP\nabla f\mu(dx)+\frac{1}{N}\int (\idx-NP^tP)JA^{-1}(x-NP^t\eta)\cdot \nabla f\mu(dx)
\\&=\frac{1}{N}\int PJA^{-1}NP^t(Px-\eta)\cdot NP\nabla f \mu(dx)+\frac{1}{N}\int PJA^{-1}(\idx-NP^tP)x\cdot NP\nabla f \mu(dx)
\\&\quad +\frac{1}{N}\int (\idx-NP^tP)JA^{-1}(x-NP^t\eta)\cdot \nabla f\mu(dx).
\end{align*}
The first term on the right hand side of the expression above can be transformed further using Lemma 21 in~\cite{GOVW09} as done for $(I)$.
\begin{align*}
&\frac{1}{N}\int PJA^{-1}(x-NP^t\eta)\cdot NP\nabla f\mu(dx)=\frac{1}{N}\int_Y \Big(\int_{Px=y}PJA^{-1}NP^t(y-\eta)\cdot NP\nabla f\mu(dx|y)\Big)\overline{\mu}dy
\\&\qquad =\frac{1}{N}\int_YPJA^{-1}NP^t(y-\eta)\cdot\Big[\nabla_Y \overline{f}(y)+NP\mathrm{cov}_{\mu(dx|y)}(f,\nabla H)\Big]\overline{\mu}(dy)
\\&\qquad = \frac{1}{N}\int_YPJA^{-1}NP^t(y-\eta)\cdot\nabla_Y \overline{f}(y)\overline{\mu}(dy)+\int_YPJA^{-1}NP^t(y-\eta)\cdot P\mathrm{cov}_{\mu(dx|y)}(f,\nabla H)\overline{\mu}(dy)
\\&=-\frac{\mathrm{Tr}(PJA^{-1}NP^t)}{N}+\int_YPJA^{-1}NP^t(y-\eta)\cdot\nabla_Y\overline{H}(y)\overline{f}\overline{\mu}(dy)
\\&\qquad\qquad+\int_YPJA^{-1}NP^t(y-\eta)\cdot P\mathrm{cov}_{\mu(dx|y)}(f,\nabla H)\overline{\mu}(dy).
\end{align*}
Since $PJA^{-1}NP^t$ is anti-symmetric, $\mathrm{Tr}(PJA^{-1}NP^t)$=0, and we obtain
\begin{align}
\label{eq: III}
(III)&=\int_YPJA^{-1}NP^t(y-\eta)\cdot\nabla_Y\overline{H}(y)\overline{f}\overline{\mu}(dy)\nonumber
\\&\qquad+\int_YPJA^{-1}NP^t(y-\eta)\cdot P\mathrm{cov}_{\mu(dx|y)}(f,\nabla H)\overline{\mu}(dy)\nonumber
\\&\qquad+\frac{1}{N}\int PJA^{-1}(\idx-NP^tP)x\cdot NP\nabla f \mu(dx)\nonumber
\\&\qquad +\frac{1}{N}\int (\idx-NP^tP)JA^{-1}(x-NP^t\eta)\cdot \nabla f\mu(dx).
\end{align}
Finally, we now transform $(IV)$.
\begin{align}
(IV)&=\int PA^{-1}NP^t(Px-\eta)\cdot \overline{J}\nabla_Y\overline{H}(\eta)f\mu(dx)+\int PA^{-1}(\idx-NP^tP)\cdot \overline{J}\nabla_Y\overline{H}(\eta)f\mu(dx)\nonumber
\\&\overset{\eqref{eq: disintegration}}{=}\int_Y PA^{-1}NP^t(y-\eta)\cdot \overline{J}\nabla_Y\overline{H}(\eta)\overline{f}\overline{\mu}(dy)+\int PA^{-1}(\idx-NP^tP)\cdot \overline{J}\nabla_Y\overline{H}(\eta)f\mu(dx)\nonumber
\\&\overset{\eqref{eq: macro A}}{=}-\int PJA^{-1}NP^t(y-\eta)\cdot\nabla_Y\overline{H}(\eta)\overline{f}\overline{\mu}(dy)+\int PA^{-1}(\idx-NP^tP)\cdot \overline{J}\nabla_Y\overline{H}(\eta)f\mu(dx).
\label{eq: IV}
\end{align}
Substituting~\eqref{eq: I}-\eqref{eq: IV} into~\eqref{eq: d/dt theta 1}, we obtain~\eqref{eq: d/dt Theta} and the lemma is proven.
\end{proof}
The following auxiliary lemma will be helpful in the sequel. The first part will be used to control the interplay between $A$ and $J$. The second and third parts are respectively (54) and (52) in~\cite{GOVW09}; we state them here for the readers' convenience. 

\begin{lemma}
\label{lemma: auxilary estimates}We have the following estimate
\begin{enumerate}
\item For every~$y\in Y$
\begin{align}
&|PJA^{-1}NP^ty|^2\leq c\langle \overline{A}^{-1}y,y\rangle\leq \frac{c}{\tau} |y|_Y^2,\label{eq: ineq1}
\\&\langle \bar{A}PJA^{-1}NP^ty, PJA^{-1}NP^ty \rangle \leq c|y|^2.\label{eq: ineq2}
\end{align}  
\item For every $x\in X$
\begin{equation}
(\idx-NP^tP)x\cdot A^{-1}(\idx-NP^tP)x\leq \frac{\gamma}{M^2}|x|^2\label{eq: ineq3}.
\end{equation} 
\item It holds that
\begin {equation}
\label{eq: cov estimate}
|NP^tP\cov_{\mu(dx|y)}(f,\nabla H)|^2\leq \gamma \frac{\kappa^2}{\rho^2}\frac{1}{M^2}\bar{f}\int \frac{1}{f}\nabla f\cdot A\nabla f\mu(dx|y).
\end{equation}
\end{enumerate}
\end{lemma}

\begin{proof}
We only need to prove the first part.

We start with~\eqref{eq: ineq1}. The first inequality is obtained using the assumption (2) and the fact that $NP^tP$ is an orthogonal projection of $X$ to $(\ker P)^\perp$ as follows.
\begin{align*}
\langle PJA^{-1}NP^ty,PJA^{-1}NP^ty\rangle&=\frac{1}{N}\langle NP^tPJA^{-1}NP^ty,JA^{-1}NP^ty\rangle
\\&\leq \frac{1}{N}\langle JA^{-1}NP^ty,JA^{-1}NP^ty\rangle
\\&=-\frac{1}{N}\langle J^2A^{-1}NP^ty,A^{-1}NP^ty\rangle
\\&\leq \frac{c}{N}\langle A^{-1}NP^ty,NP^ty\rangle \quad\text{(used assumption~\eqref{asumpt: weak anti-symmetric} here)}
\\&=c\langle NPA^{-1}P^ty,y\rangle
\\&=c\langle\overline{A}^{-1}y,y\rangle.
\end{align*}
Now we prove the second one. Since $\tau$ is a lower bound on the spectral value of $A$, $\frac{1}{\tau}$ is an upper bound on that of $A^{-1}$. Hence
\begin{equation*}
\langle \overline{A}^{-1}y,y\rangle=\langle PA^{-1}NP^ty,y\rangle=\frac{1}{N}\langle A^{-1}NP^ty,NP^ty \rangle\leq \frac{1}{N\tau}\langle NP^ty,NP^ty\rangle=\frac{1}{\tau}|y|_Y^2.
\end{equation*}
Next, we prove~\eqref{eq: ineq2}. By duality, we have
\begin{align}
&\langle \bar{A}PJA^{-1}NP^ty, PJA^{-1}NP^ty \rangle = \underset{z}{\sup} \hspace{1mm}\{ 2\langle PJA^{-1}NP^ty, z \rangle - \langle \bar{A}^{-1}z, z \rangle\} \notag  \\
&\overset{\eqref{eq: ineq1}}{\leq} \underset{z}{\sup} \hspace{1mm} \{2\langle y, PJA^{-1}NP^tz \rangle - c^{-1}|PJA^{-1}NP^tz|^2\} \notag \\
&\leq \underset{z}{\sup} \hspace{1mm} \{2\langle y, z \rangle - c^{-1}|z|^2\} \notag \\
&\leq c|y|^2 \notag.
\end{align}
\end{proof}

\begin{lemma} \label{lem : bounds entropy production}
If $f(t,x)$ and $\eta(t)$ satisfy the assumptions of theorem~\ref{theo: abstract theorem}, then for any $T<\infty$ we have
\begin{equation}
\label{eq: bound of Fisher term}
\int_0^T\int\frac{1}{f}\nabla f\cdot A\nabla f(t,x)\mu(dx)dt=\int \Phi(f(0,x))\mu(dx)-\int \Phi(f(T,x))\mu(dx);
\end{equation}
\begin{equation}
\label{eq: bound of macro Fisher term}
\int_0^T\langle\overline{A}\nabla_Y\overline{H}(\eta),\nabla_Y\overline{H}(\eta)\rangle dt\leq 2(H(\eta_0)-H(\eta_T))+CT(1+e^{CT}\overline{H}(\eta_0)),
\end{equation}
where $C>0$ is a constant;
\begin{equation}
\label{eq: bound of second moment}
\left(\int |x|^2f(t,x)\mu(dx)\right)^\frac{1}{2}\leq \left(\int|x|^2\mu(dx)\right)^\frac{1}{2}+\left(\frac{2}{\rho}\int\Phi(f(0,x))\mu(dx)\right)^\frac{1}{2}.
\end{equation}
\end{lemma}

\begin{proof}
The proof of this lemma is similar to that of proposition 24 in~\cite{GOVW09}.
We prove~\eqref{eq: bound of Fisher term} first. We have
\begin{align}
\frac{d}{dt}\int \Phi(f(t,x))\mu(dx)&=\int (\log f+1)\partial_t(f\mu)\notag
\\&=\int (\log f+1)\mathrm{div}(\mu(A+J)\nabla f)\notag
\\&=-\int (A+J)\nabla f\cdot \frac{\nabla f}{f}\mu(dx)\notag
\\&=-\int  \frac{1}{f}A\nabla f\cdot\nabla f\mu(dx)\qquad\text{(since $J$ is anti-symmetric)}\label{eq: entropy decrease}.
\end{align}
Thus \eqref{eq: bound of Fisher term} follows. Next we prove~\eqref{eq: bound of macro Fisher term}. We have
\begin{align*}
\frac{d}{dt}\overline{H}(\eta(t))&=\langle\dot{\eta}(t),\nabla_Y\overline{H}(\eta)\rangle
\\&\overset{\eqref{eq: macro eqn}}{=}-\langle\overline{A}\nabla_Y\overline{H}(\eta),\nabla_Y\overline{H}(\eta)\rangle-\langle\overline{J}\nabla_Y\overline{H}(\eta),\nabla_Y\overline{H}(\eta)\rangle
\\&= -\langle \bar{A}\nabla_Y \bar{H}(\eta), \nabla_Y \bar{H}(\eta) \rangle - \langle \bar{A}PJA^{-1}NP^t \nabla_Y \bar{H}(\eta), \nabla_Y \bar{H}(\eta) \rangle \notag \\
&\leq  \frac{1}{2}\langle \bar{A}PJA^{-1}NP^t \nabla_Y \bar{H}(\eta), PJA^{-1}NP^t \nabla_Y \bar{H}(\eta)\rangle-\frac{1}{2}\langle\overline{A}\nabla_Y\overline{H}(\eta),\nabla_Y\overline{H}(\eta)\rangle \notag
\\&\overset{\eqref{eq: ineq2}}{\leq} \frac{c}{2}|\nabla_Y \bar{H}(\eta)|^2-\frac{1}{2}\langle\overline{A}\nabla_Y\overline{H}(\eta),\nabla_Y\overline{H}(\eta)\rangle. \notag
\end{align*}
In this computation, we used the assumption that $A$ and $J$ commute. Therefore
\begin{align*}
\frac{d}{dt}\overline{H}(\eta(t)) &+ \frac{1}{2}\langle\overline{A}\nabla_Y\overline{H}(\eta),\nabla_Y\overline{H}(\eta)\rangle \leq\frac{c}{2}|\nabla \bar{H}(\eta)|^2
\\&\leq C(|\eta|^2+1) \leq C(\overline{H}(\eta)+1).
\end{align*}
In the above estimate, $C>0$ is a general constant. Note that we have used assumption~\eqref{asumpt: bounded of Hess macro H}. The above Gronwall-type inequality implies that for every $t\geq 0$, we have $\overline{H}(\eta(t))\leq e^{CT}(\overline{H}(\eta_0) + 1)$, and 
\[\int_0^T\langle\overline{A}\nabla_Y\overline{H}(\eta),\nabla_Y\overline{H}(\eta)\rangle dt\leq CTe^{CT}(\bar{H}(\eta_0) +1).
\]
According to Lemma 26 in \cite{GOVW09}, since $\mu$ satisfies LSI($\rho$), we have
$$\left(\int{|x|^2f(t,x)\mu(dx)}\right)^{1/2} \leq \left(\int{|x|^2\mu(dx)}\right)^{1/2} + \left(\frac{2}{\rho}\int{\Phi(f(t,x))\mu(dx)}\right)^{1/2}.$$
By~\eqref{eq: entropy decrease}, $\int\Phi(f(t,x))\mu(dx)$ is non-increasing in $t$, and~\eqref{eq: bound of second moment} immediately follows. 
\end{proof}
\begin{lemma}
We have the following estimate
\begin{align}
&\frac{d}{dt}\Theta(t)-\frac{8c\Lambda^2}{\lambda}\Theta(t)+\frac{\lambda}{8}\int|y-\eta|^2\overline{f}\overline{\mu}(dy)\notag
\\&\quad \leq \frac{M}{N}+\frac{4c\gamma\Lambda^2}{2\lambda NM^2}\int|x|^2f\mu(dx)\notag
\\&\qquad+\left(\frac{\gamma \kappa^2}{2\lambda\rho^2M^2}+\frac{2c\gamma\kappa^2}{\tau\lambda\rho^2M^2}+\frac{4\gamma c}{\lambda \tau M^2}\right)\int \frac{1}{Nf}\nabla f\cdot A\nabla f\mu(dx)\notag
\\&\qquad+\frac{\sqrt{\gamma}}{M}\left(\int \frac{1}{N}|x|^2f\mu(dx)\right)^\frac{1}{2}\left[\left(1+\sqrt{\frac{c}{\tau}}+\frac{\sqrt{2c\gamma}}{M}\right)\left(\int\frac{1}{Nf}\nabla f\cdot A\nabla f\mu(dx)\right)^\frac{1}{2}\right.\notag\\
&\left.\hspace*{6cm}+\left(1+\sqrt{\frac{c}{\tau}}\right)\left(\overline{A}\nabla_Y\overline{H}(\eta)\cdot\nabla_Y\overline{H}(\eta)\right)^\frac{1}{2}\right].\label{eq: estimate derivative of Theta}
\end{align}
\end{lemma}
\begin{proof}
We estimate each term in~\eqref{eq: d/dt Theta}. The 2nd, 4th and 5th terms are already estimated in~\cite{GOVW09} (these are respectively equations (50), (53) and (55) in \cite{GOVW09}). We get
\begin{align}
&-\int_Y(y-\eta)\cdot(\nabla_Y\overline{H}(y)-\nabla_Y\overline{H}(\eta))\overline{f}\overline{\mu}dy\leq -\lambda\int |y-\eta|_Y^2\overline{f}\overline{\mu}dy,\label{eq: 2nd term}
\\&\left|\int(y-\eta)\cdot P \mathrm{cov}_{\mu(dx|y)}(f,\nabla H)\overline{\mu}dy\right|\leq \frac{\gamma\kappa^2}{2\lambda\rho^2M^2}\int\frac{1}{Nf}\nabla f\cdot A\nabla f\mu(dx)+\frac{\lambda}{2}\int |y-\eta|_Y^2\overline{f}\overline{\mu}(dy),\label{eq: 4th term}
\\&\left|\frac{1}{N}\int (\idx-NP^tP)x\cdot \nabla f\mu(dx)\right|\leq \left(\frac{\gamma}{M^2}\int\frac{1}{Nf}\nabla f\cdot A\nabla f\mu(dx)\cdot\int\frac{1}{N}|x|^2f\mu(dx)\right)^\frac{1}{2}.\label{eq: 5th term}
\end{align}
We estimate the 3rd term.
Since
\begin{align*}
&|PJA^{-1}NP^t(y-\eta)|\cdot|\nabla_Y\overline{H}(y)-\nabla_Y\overline{H}(\eta)| \leq \Lambda|y-\eta|\cdot |PJA^{-1}NP^t(y-\eta)| \notag
\\&\qquad \overset{\eqref{eq: ineq1}}{\leq} \Lambda|y-\eta|\sqrt{c\langle\overline{A}^{-1}(y-\eta),y-\eta\rangle}\notag
\\&\qquad \leq\frac{\lambda}{8}|y-\eta|^2+\frac{2c\Lambda^2}{\lambda}\langle\overline{A}^{-1}(y-\eta),y-\eta\rangle,
\end{align*}
we have
\begin{align}
&\left|\int_YPJA^{-1}NP^t(y-\eta)\cdot(\nabla_Y\overline{H}(y)-\nabla_Y\overline{H}(\eta))\overline{f}\overline{\mu}(dy)\right|\nonumber
\\&\qquad\leq \int_Y |PJA^{-1}NP^t(y-\eta)||\nabla_Y\overline{H}(y)-\nabla_Y\overline{H}(\eta)|\overline{f}\overline{\mu}(dy)\nonumber
\\&\qquad \leq\frac{\lambda}{8}\int_Y|y-\eta|^2\overline{f}\overline{\mu}(dy)+\frac{2c\Lambda^2}{\lambda}\int_Y\langle\overline{A}^{-1}(y-\eta),y-\eta\rangle\overline{f}\overline{\mu}(dy)\notag
\\&\qquad=\frac{\lambda}{8}\int_Y|y-\eta|^2\overline{f}\overline{\mu}(dy)+\frac{2c\Lambda^2}{\lambda}\frac{1}{N}\int_X\langle A^{-1}NP^t(Px-\eta),NP^t(Px-\eta)\rangle f\mu(dx)\notag
\\&\qquad\leq\frac{\lambda}{8}\int_Y|y-\eta|^2\overline{f}\overline{\mu}(dy)+\frac{2c\Lambda^2}{\lambda}\frac{2}{N}\int_X\langle A^{-1}(x-NP^t\eta),(x-NP^t\eta)\rangle f\mu(dx)\notag
\\&\hspace{20mm}+\frac{2c\Lambda^2}{\lambda}\frac{2}{N}\int_X\langle A^{-1}(\idx-NP^tP)x,(\idx-NP^tP)x\rangle f\mu(dx)\notag
\\&\qquad\leq \frac{\lambda}{8}\int_Y|y-\eta|^2\overline{f}\overline{\mu}(dy)+\frac{8c\Lambda^2}{\lambda}\Theta(t)+\frac{4c\gamma\Lambda^2}{\lambda NM^2}\int|x|^2f\mu(dx),
\label{eq: 3rd term}
\end{align}
where we used \eqref{eq: ineq3} to get the last inequality. Next we estimate the 6th term.
\begin{align*}
&\int \overline{A}\nabla_Y\overline{H}(\eta)\cdot PA^{-1}(\idx-NP^tP)xf\mu(dx)
\\&\quad=\int P^t\overline{A}\nabla_Y\overline{H}(\eta)\cdot A^{-1}(\idx-NP^tP)xf\mu(dx)
\\&\quad\leq\left(\int P^t\overline{A}\nabla_Y\overline{H}(\eta)\cdot A^{-1}NP^t\overline{A}\nabla_Y\overline{H}(\eta)f\mu(dx)\right)^\frac{1}{2}\left(\frac{1}{N}\int (\idx-NP^tP)x\cdot A^{-1}(\idx-NP^tP)xf\mu(dx)\right)^\frac{1}{2}.
\end{align*}
Since
\begin{equation*}
P^t\overline{A}\nabla_Y\overline{H}(\eta)\cdot A^{-1}NP^t\overline{A}\nabla_Y\overline{H}(\eta)=\overline{A}\nabla_Y\overline{H}(\eta)\cdot PA^{-1}NP^t\overline{A}\nabla_Y\overline{H}(\eta)=\overline{A}\nabla_Y\overline{H}(\eta)\cdot\nabla_Y\overline{H},
\end{equation*}
and from~\eqref{eq: ineq3},
we have
\begin{equation}
\int \overline{A}\nabla_Y\overline{H}(\eta)\cdot PA^{-1}(\idx-NP^tP)xf\mu(dx)\leq \left(\overline{A}\nabla_Y \overline{H}(\eta)\cdot\nabla_Y\overline{H}(\eta)\right)^\frac{1}{2}\left(\frac{\gamma}{NM^2}\int |x|^2f\mu(dx),\right)^\frac{1}{2}.\label{eq: 6th term}
\end{equation}
Next, we estimate the 7th term.
\begin{align}
&\left|\int PJA^{-1}NP^t(y-\eta)\cdot P\mathrm{cov}_{\mu(dx|y)}(f,\nabla H)\overline{\mu}(dy)\right|\nonumber
\\&\qquad \leq \left(\int |PJA^{-1}NP^t(y-\eta)|^2\overline{f}\overline{\mu}(dy)\cdot\int\frac{1}{\overline{f}}|P\mathrm{cov}_{\mu(dx|y)}(f,\nabla H)|^2_Y\overline{\mu}(dy)\right)^\frac{1}{2}\nonumber
\\&\qquad \overset{\eqref{eq: ineq1},\eqref{eq: cov estimate}}{\leq} \left(\frac{c}{\tau}\gamma\frac{\kappa^2}{\rho^2}\frac{1}{M^2}\int|y-\eta|^2\overline{f}\overline{\mu}(dy)\int\frac{1}{Nf}\nabla f\cdot A\nabla f\mu(dx)\right)^\frac{1}{2}\nonumber
\\&\qquad \leq \frac{2c\gamma\kappa^2}{\tau\lambda\rho^2M^2}\int\frac{1}{Nf}\nabla f\cdot A\nabla f\mu(dx)+\frac{\lambda}{8}\int |y-\eta|^2\overline{f}\overline{\mu}(dy).\label{eq: 7th term}
\end{align}
For the 8th term, we have
\begin{align*}
&\int PJA^{-1}(\idx-NP^tP)x\cdot P\nabla f\mu(dx)=\frac{1}{N}\int NP^tPJA^{-1}(\idx-NP^tP)x\nabla f\mu(dx)
\\&\leq \left(\int \frac{1}{N}NP^tPJA^{-1}(\idx-NP^tP)x \cdot A^{-1}NP^tPJA^{-1}(\idx-NP^tP)xf\mu(dx)\right)^\frac{1}{2}\left(\int \frac{1}{Nf}\nabla f\cdot A\nabla f\mu(dx)\right)^\frac{1}{2}.
\end{align*}
Since
\begin{align*}
&\langle NP^tPJA^{-1}(\idx-NP^tP)x,A^{-1}NP^tPJA^{-1}(\idx-NP^tP)x\rangle
\\&\qquad \overset{(viii)}{\leq} \frac{1}{\tau}\langle NP^tPJA^{-1}(\idx-NP^tP)x,NP^tPJA^{-1}(\idx-NP^tP)x\rangle
\\&\qquad \leq \frac{1}{\tau}\langle JA^{-1}(\idx-NP^tP)x,JA^{-1}(\idx-NP^tP)x\rangle
\\&\qquad=\frac{1}{\tau}\langle -J^2A^{-1}(\idx-NP^tP)x,A^{-1}(\idx-NP^tP)x\rangle
\\&\qquad\overset{(ix)}{\leq} \frac{c}{\tau}\langle (\idx-NP^tP)x,A^{-1}(\idx-NP^tP)x\rangle
\\&\qquad\overset{\eqref{eq: ineq3}}{\leq} \frac{c\gamma}{\tau M^2}|x|^2,
\end{align*}
we obtain
\begin{equation}
\int PJA^{-1}(\idx-NP^tP)x\cdot P\nabla f\mu(dx)
\leq \left(\frac{c\gamma}{\tau M^2}\int \frac{1}{N}|x|^2f\mu(dx)\cdot \int \frac{1}{Nf}\nabla f\cdot A\nabla f\mu(dx)\right)^\frac{1}{2}\label{eq: 8th term}.
\end{equation}
Next we estimate the 9th term. Set $z=JA^{-1}(x-NP^t\eta)$. We have
\begin{align*}
&\frac{1}{N}\int (\idx-NP^tP)JA^{-1}(x-NP^t\eta)\cdot \nabla f\mu(dx)=\frac{1}{N}\int (\idx-NP^tP)z\cdot \nabla f\mu(dx)
\\&\leq \left(\int\frac{1}{N}(\idx-NP^tP)z\cdot A^{-1}(\idx-NP^tP)zf\mu(dx)\int \frac{1}{Nf}\nabla f\cdot A\nabla f \mu(dx)\right)^\frac{1}{2}
\\&\overset{\eqref{eq: ineq3}}{\leq} \left(\frac{\gamma}{M^2}\int \frac{1}{Nf}\nabla f\cdot A\nabla f\mu(dx)\int\frac{1}{N}|z|^2f\mu(dx)\right)^\frac{1}{2}.
\end{align*}
We estimate the second integral inside the parentheses. It holds that
\begin{align*}
|z|^2&=\langle JA^{-1}(x-NP^t\eta),JA^{-1}(x-NP^t\eta)\rangle =\langle -J^2A^{-1}(x-NP^t\eta),A^{-1}(x-NP^t\eta)\rangle
\\&\overset{\eqref{asumpt: weak anti-symmetric}, (viii)}{\leq}c\langle A^{-1}(x-NP^t\eta),x-NP^t\eta\rangle
\\&\leq 2c\left(\langle A^{-1}NP^t(Px-\eta),NP^t(Px-\eta)\rangle+\langle A^{-1}(\idx-NP^tP)x,(\idx-NP^tP)x\rangle\right)
\\&\overset{\eqref{eq: ineq3}}{\leq} 2c\left(\frac{1}{\tau}|NP^t(Px-\eta)|^2+\frac{\gamma}{M^2}|x|^2\right) =2c\left(\frac{N}{\tau}|Px-\eta|^2+\frac{\gamma}{M^2}|x|^2\right).
\end{align*}
Therefore,
\begin{align}
&\frac{1}{N}\int (\idx-NP^tP)JA^{-1}(x-NP^t\eta)\cdot \nabla f\mu(dx)\nonumber
\\&\qquad\leq \left(\frac{\gamma}{M^2}\int\frac{1}{Nf}\nabla f\cdot A\nabla f\mu(dx)\right)^\frac{1}{2}\left(\frac{2c}{\tau}\int |y-\eta|^2\overline{f}\overline{\mu}(dy)+\frac{2c\gamma}{M^2N}\int |x|^2f\mu(dx)\right)^\frac{1}{2}\nonumber
\\&\qquad\leq \frac{4\gamma c}{M^2\lambda \tau}\int\frac{1}{Nf}\nabla f\cdot A\nabla f\mu(dx)+\frac{\lambda}{8}\int |y-\eta|^2\overline{f}\overline{\mu}(dy)\nonumber
\\&\qquad\qquad+\left(\frac{\gamma}{M^2}\int\frac{1}{Nf}\nabla f\cdot A\nabla f\mu(dx)\right)^\frac{1}{2}\left(\frac{2c\gamma}{M^2N}\int|x|^2f\mu(dx)\right)^\frac{1}{2}\label{eq: 9th term}.
\end{align}
Finally, we estimate the 10th term.
Since
\begin{align*}
&\langle PA^{-1}NP^t\bar{J}\nabla \bar{H}(\eta), \bar{J}\nabla \bar{H}(\eta) \rangle = \langle PJA^{-1}NP^t\nabla \bar{H}(\eta), \bar{A}PJA^{-1}NP^t\nabla \bar{H}(\eta) \rangle \notag \\
&\overset{(x)}{\leq} c|\nabla \bar{H}(\eta)|^2 \leq \frac{c}{\tau}\langle\bar{A}\nabla_Y\bar{H}(\eta),\nabla_Y\bar{H}(\eta)\rangle,
\end{align*}
we have
\begin{align}
&\left|\int A^{-1}(\idx-NP^tP)x\cdot P^t\overline{J}\nabla_Y\overline{H}(\eta)f\mu(dx)\right|\notag
\\&\qquad\leq\left(\int P^t\overline{J}\nabla_Y\overline{H}(\eta)\cdot A^{-1}NP^t\overline{J}\nabla_Y\overline{H}(\eta)f\mu(dx)\right)^\frac{1}{2}\left(\int\frac{1}{N}(\idx-NP^tP)x\cdot A^{-1}(\idx-NP^tP)xf\mu(dx)\right)^\frac{1}{2}\nonumber
\\&\qquad\leq\left(\frac{c}{\tau}\overline{A}\nabla_Y\overline{H}(\eta)\cdot\nabla_Y\overline{H}(\eta)\right)^\frac{1}{2}\left(\frac{\gamma}{NM^2}\int|x|^2f\mu(dx)\right)^\frac{1}{2}\label{eq: the 10th term}.
\end{align}
Summing up from~\eqref{eq: 2nd term} to~\eqref{eq: the 10th term}, we obtain \eqref{eq: estimate derivative of Theta}
\end{proof}
\begin{proof}[\textbf{Proof of Theorem~\ref{theo: abstract theorem}}]
Denote by $R(t)$ the right hand side of~\eqref{eq: estimate derivative of Theta}. Set $D=\frac{8c\Lambda^2}{\lambda}$.  For any $0<t\leq T$, we have
\begin{align}
\frac{d}{dt}\left(e^{-Dt}\Theta(t)\right)+  e^{-DT}\frac{\lambda}{8}\int |y-\eta|^2\bar{f}\bar{\mu}(dy)&\leq \frac{d}{dt}\left(e^{-Dt}\Theta(t)\right)+  e^{-Dt}\frac{\lambda}{8}\int |y-\eta|^2\bar{f}\bar{\mu}(dy)\notag
\\&\leq e^{-Dt} R(t)\leq R(t).\label{eq: differential form}
\end{align}
 Integrating~\eqref{eq: differential form} with respect to time, for any $0<t\leq T$, we have
 \begin{equation}
e^{-DT}\Theta(t)+\frac{\lambda}{8}e^{-DT}\int_t^T{\int{|y-\eta|^2\bar{f}\bar{\mu}(dy)}dt} \leq \Theta(0)+\int_0^T R(t)dt.
 \end{equation}
 It follows that for any $T>0$
 \begin{equation}
 \label{eq:temp}
 \max \left\{\sup_{t\in(0,T)} \Theta(t),\frac{\lambda}{8}\int_0^T\int_Y|y-\eta|^2\bar{f}\bar{\mu}(dy)\right\}\leq e^{DT}\left(\Theta(0)+\int_0^TR(t)dt\right).
 \end{equation}
It remains to take care of each term in the right hand side of~\eqref{eq:temp}. Recall that $R$ is the right hand side of~\eqref{eq: estimate derivative of Theta}. Let $a=1+\sqrt{\frac{c}{\tau}}+\frac{\sqrt{2c\gamma}}{M},b=1+\sqrt{\frac{c}{\tau}}$. We have the following estimates
\begin{align*}
&\int_0^T\int\frac{1}{N}|x|^2f(t,x)\mu(dx)dt\overset{\eqref{eq: bound of second moment}}{\leq}2\left(\alpha+\frac{2C_1}{\hat{\rho}}	\right)T;
\\&\int_0^T\int\frac{1}{Nf}\nabla\cdot A\nabla f\mu(dx)dt\overset{\eqref{eq: bound of Fisher term}}{\leq }C_1;
\\&\int_0^T\left(\int\frac{1}{N}|x|^2f\mu(dx)\right)^\frac{1}{2}\left(a\left(\int\frac{1}{Nf}\nabla f\cdot A\nabla f\mu(dx)\right)^\frac{1}{2}+b\left(\bar{A}\nabla_Y\bar{H}(\eta)\cdot\nabla_Y\bar{H}(\eta)\right)^\frac{1}{2}\right) dt
\\&\qquad\leq \left(\int_0^T\int\frac{1}{N}|x|^2f\mu(dx)dt\right)^\frac{1}{2}\left(a\left(\int_0^T\int\frac{1}{N}\nabla f\cdot A\nabla f\mu(dx)dt\right)^\frac{1}{2}+b\left(\int_0^T\bar{A}\nabla_Y\bar{H}(\eta)\cdot\nabla_Y\bar{H}(\eta)dt\right)^\frac{1}{2}\right)
\\&\qquad\leq\sqrt{2T}\left(\alpha+\frac{2C_1}{\hat{\rho}}\right)^\frac{1}{2}\left(a\sqrt{C_1}+\sqrt{2}b(H(\eta_0)-H(\eta_T))+CT(1+e^{CT}\overline{H}(\eta_0))^\frac{1}{2}\right).
\end{align*}
Substituting these estimate to~\eqref{eq:temp} concludes the proof of Theorem~\ref{theo: abstract theorem}.
\end{proof}

\subsection{Sketch of proof of Theorem \ref{theo: abstract entropy}}

In this section, we give the main arguments of the proof of Theorem \ref{theo: abstract entropy}, which exactly follows the method of \cite{Fat13}. Recall that the local Gibbs state is given by $G(x)\mu(dx)$, where $G(x) = Z^{-1}\exp(NP^t\bar{H}(\eta)\cdot x)$.

\begin{itemize}
\item First, we decompose the relative entropy with respect to the local Gibbs state into a macroscopic component and a fluctuations component. Since $G(x)$ only depends  on the macroscopic profile $y = Px$, we have
$$\Ent_{G\mu}(f\mu) = \Ent_{\bar{G}\bar{\mu}}(\bar{f}\bar{\mu}) + \int_Y{\Ent_{\mu(dx|y)}(f\mu)\bar{G}(y)\bar{\mu}(dy)}.$$

\item The fluctuations component $\int_0^T{ \int_Y{\Ent_{\mu(dx|y)}(f\mu)\bar{G}(y)\bar{\mu}(dy)}dt}$ can be bounded using the logarithmic Sobolev inequality for $\mu(dx|y)$, assumption (vi) and \eqref{eq: bound of Fisher term} in Lemma \ref{lem : bounds entropy production}.

\item For the macroscopic component, since $\bar{G}\bar{\mu}$ is log-concave, we can use the HWI inequality of \cite{OV}, which states that
$$\Ent_{\bar{G}\bar{\mu}}(\bar{f}\bar{\mu}) \leq W_2(\bar{f}\bar{\mu}, \bar{G}\bar{\mu})\sqrt{I_{\bar{G}\bar{\mu}}(\bar{f}\bar{\mu})},$$
where the Wasserstein distance $W_2$ is taken with respect to the norm $|\cdot |_Y$, and $I$ is the Fisher information
$$I_{\bar{G}\bar{\mu}}(\bar{f}\bar{\mu}) :=\int \frac{|\nabla (\bar{f}/\bar{G})|^2}{\bar{f}/\bar{G}}\bar{G}d\bar{\mu}. $$
As a consequence, to obtain convergence in relative entropy, we only require convergence in Wasserstein distance and a bound on the Fisher information.

\item We already have a bound on $\int_0^T{W_2(\bar{f}\bar{\mu}, \delta_{\eta(t)})^2dt}$ from Theorem \ref{theo: abstract theorem}. Moreover,
$$W_2(\bar{G}\bar{\mu}, \delta_{\eta})^2 \leq \frac{M}{\lambda N}$$
by Proposition 4.1 of \cite{Fat13}. A bound on $\int_0^T{W_2(\bar{f}\bar{\mu}, \bar{G}\bar{\mu})dt}$ immediately follows from the triangle inequality.

\item Finally, the time-integral of the Fisher information can be bounded using the bounds on the entropy production of Lemma \ref{lem : bounds entropy production}. This concludes the proof of (a).

\item (b) can be deduced from (a) using elementary inequalities and the bound
\begin{align}
&\left| \frac{1}{N}\int{\Phi(G^{\eta})d\mu} - \bar{H}(\eta) \right| \notag \\
&\leq \frac{(M-1)}{2N} \max\left(\left|\log\left(\frac{\Gamma(Y,|\cdot|_Y)^{2/(M-1)}}{\Lambda N}\right)\right|,\left|\log\left(\frac{\Gamma(Y,|\cdot|_Y)^{2/(M-1)}}{\lambda N}\right)\right| \right) \notag \\
& \hspace{5mm} +\sqrt{\frac{M}{\lambda N}}|\nabla \bar{H}(\eta)|_Y, \notag
\end{align}
which was proven in Proposition 4.1 of \cite{Fat13}.
\end{itemize}

\section{Application to weakly asymmetric Kawasaki dynamics}
In this section, we prove Theorem~\ref{theo: concrete hydro}. First, we give a precise definition of the notion of weak solution to the limiting equation~\eqref{eq: limiting eqn}.
\begin{definition}
$\zeta=\zeta(t,\theta)$ is called a weak solution of~\eqref{eq: limiting eqn} on $[0,T]\times\T^1 $ if
\begin{equation}
\label{eq: regularity of weak sol}
\zeta\in L_t^\infty(L_{\theta}^2),\quad\frac{\partial \zeta}{\partial t}\in L_t^2(H_{\theta}^{-1}),\quad \varphi'(\zeta)\in L_t^2(L_{\theta}^2),
\end{equation}
and
\begin{equation}
\label{eq: weak form}
\left\langle g,\frac{\partial \zeta}{\partial t}\right\rangle_{H^{-1}}=-\int_{\T^1}g\varphi'(\zeta)\,d\theta+\int_{\T^1}G\varphi'(\zeta)\,d\theta,~\text{for all}~g\in L^2(\T^1),~\text{for almost every }~ t\in[0,T],
\end{equation}
where $G$ is the (unique up to a set of Lebesgue measure $0$) function on the torus
such that $\int_{\T^1}G\,d\theta =0 $ and $G'= g$.
\end{definition}

As in Corollary \ref{cor: abstract theorem}, consider a sequence $\{M_{\ell}, N_{\ell}\}_{\ell = 1}^{\infty}$ such that 
$$M_{\ell} \uparrow \infty; \hspace{3mm} N_{\ell} \uparrow \infty; \hspace{3mm} \frac{N_{\ell}}{M_{\ell}} \uparrow \infty.$$
Let $\bar{\eta}^{\ell}_0$ be a step-function approximation of $\zeta_0$, such that 
\begin{equation}
||\bar{\eta}^{\ell}_0 - \zeta_0 ||_{L^2} \underset{\ell \uparrow \infty}{\longrightarrow} \hspace{1mm} 0.
\end{equation}
Consider $\eta^{\ell}$ the solutions to 
$$\frac{d\eta^{\ell}}{dt} = -(\bar{A} + \bar{J})\nabla_Y \bar{H}(\eta^{\ell}), \hspace{3mm} \eta^{\ell}(0) = \eta^{\ell}_0.$$

To obtain Theorem \ref{theo: concrete hydro} from Corollary \ref{cor: abstract theorem}, we shall need to study the convergence of the sequence $\eta^{\ell}$. It is given by the following result: 

\begin{proposition} \label{prop: convergence eta}
With the notations above, the sequence of step functions $\bar{\eta}^{\ell}$ converges strongly in $L_t^{\infty}(H_{\theta}^{-1})$ to the unique weak solution of ~\eqref{eq: limiting eqn} with initial condition $\zeta_0$.
\end{proposition}

The key estimate allowing us to pass to the limit is the fact that, when $N$ goes to infinity, the Euclidean product associated to $A^{-1}$ behaves like the $H^{-1}$ norm. This is the content of the following lemma: 

\begin{lemma}
\label{lem: norms}
There exists $C < + \infty$ such that, for any $x \in X$, if $\bar{x}$ is the associated step function (defined by (\ref{def: bar x})), then
$$\frac{1}{C}||\bar{x}||_{H^{-1}}^2 \leq \frac{1}{N}\langle A^{-1}x, x \rangle \leq C||\bar{x}||_{H^{-1}}^2.$$
Moreover, if $\bar{x}$ is bounded in $L^2$, then 
$$\left| ||\bar{x}||_{H^{-1}}^2 -  \frac{1}{N}\langle A^{-1}x, x \rangle\right| \leq \frac{C}{N}.$$
\end{lemma}

\noindent These estimates have been proven in Section 6.3 of \cite{GOVW09}. 

\noindent We delay the proof of Proposition \ref{prop: convergence eta}, and first prove Theorem \ref{theo: concrete hydro}

\begin{proof}[Proof of Theorem \ref{theo: concrete hydro}]
Our aim is to apply Corollary \ref{cor: abstract theorem}. To do this, we need to check that assumptions (i) to (x) hold with uniform constants. Assumptions (i) to (vii) have been checked in \cite{GOVW09}, and assumption (viii) in \cite{Fat13}. Assumption (x) can be immediately checked by the direct computation of $JA$ and $AJ$. Finally, it is easy to see that for any $x \in X$, we have
\begin{align}
\langle -J^2x, x \rangle &= |Jx|^2 \notag \\
&= \frac{N^2}{4}\underset{i = 1}{\stackrel{N}{\sum}} \hspace{1mm} (x_{i+1} - x_{i-1})^2 \notag \\
&\leq \frac{N^2}{4}\underset{i = 1}{\stackrel{N}{\sum}} 2(x_{i+1} - x_i)^2 + 2(x_i - x_{i-1})^2 \notag \\
&= N^2 \underset{i = 1}{\stackrel{N}{\sum}} (x_{i+1} - x_i)^2 \notag \\
&= \langle Ax, x \rangle, 
\end{align}
and therefore assumption (ix) holds with $c = 1$.

Applying Corollary \ref{cor: abstract theorem}, we get 
$$\underset{\ell \uparrow \infty}{\lim} \hspace{1mm} \underset{0 \leq t \leq T}{\sup} \hspace{1mm} \int{\langle (x - NP^t\eta^{\ell}(t)), A^{-1}(x - NP^t\eta^{\ell}(t))\rangle f(t,x)\mu(dx)} = 0.$$
By Lemma \ref{lem: norms}, this implies
$$\underset{\ell \uparrow \infty}{\lim} \hspace{1mm} \underset{0 \leq t \leq T}{\sup} \hspace{1mm} \int{||\bar{x} - \bar{\eta}^{\ell}(t)||_{H^{-1}}^2 f(t,x)\mu(dx)} = 0.$$
Applying Proposition \ref{prop: convergence eta} and using the triangle inequality then concludes the proof.
\end{proof}


We now turn to the proof of Proposition \ref{prop: convergence eta}. It is based on the following six lemmas, and closely follows the method of \cite{GOVW09}, with additional arguments to take into account the extra first-order term that appear due to the addition of $J$ to the dynamics.

\begin{lemma} 
\label{lem: variational}
Assume $\bar H$ is convex. Then $\eta$ satisfies~\eqref{eq: macro eqn} with initial condition $\eta(0)=\eta_0$ if and only if
\begin{equation}
\label{eq: variational form of the macro eqn}
2\int_0^T\bar{H}(\eta)\beta(t)dt\leq \int_0^T\left[\bar{H}(\eta+g)+\bar{H}(\eta-PA^{-1}NJP^t g)\right]\beta(t)dt-\int_0^T\langle g,(\bar{A})^{-1}\eta\rangle_Y\dot{\beta}(t)dt,
\end{equation}
for all $g\in Y$ and smooth $\beta\colon [0,T]\rightarrow [0,\infty)$.\\
\vspace{3mm}
Similarly, assume that $\varphi$ is convex. Then $\zeta$ satisfies~\eqref{eq: weak form} if and only if
\begin{align}
\label{eq: variational form of the limiting eqn}
& 2\int_0^T\int_{\T^1}\varphi(\zeta(t,\theta))\beta(t)\,d\theta\,dt\notag
\\&\qquad\leq \int_0^T\int_{\T^1}\left[\varphi(\zeta(t,\theta)+g(\theta))+\varphi(\zeta(t,\theta)-G(\theta))\right]\beta(t)\,d\theta\,dt-\int_0^T\langle g(\cdot),\zeta(t,\cdot)\rangle_{H^{-1}}\dot{\beta}(t)dt,
\end{align}
for all $g\in L^2(\T^1)$ and smooth $\beta\colon[0,T]\rightarrow [0,\infty)$, where $G$ is the (unique up to a set of Lebesgue measure $0$) function on the torus
such that $\int_{\T^1}G\,d\theta =0 $ and $G'= g$.
\end{lemma}

\begin{proof}
The proof of this Lemma is modified from that of Lemma 36 in~\cite{GOVW09}.
We show that~\eqref{eq: macro eqn} is equivalent to~\eqref{eq: variational form of the macro eqn}. The equivalence of~\eqref{eq: weak form} and~\eqref{eq: variational form of the limiting eqn} follows analogously. \\
The weak form of~\eqref{eq: macro eqn} is given by
\begin{equation}
\label{eq: weak form of the macro eqn}
\int_0^T\langle g,(\bar{A})^{-1}\eta\rangle_Y\dot{\beta}(t)dt=\int_0^T\left[\langle g,\nabla_Y\bar{H}(\eta)\rangle_Y-\langle PA^{-1}NJP^t g,\nabla_Y\bar{H}(\eta)\rangle_Y\right]\beta(t)\,dt,
\end{equation}
for all $g\in Y$ and smooth $\beta\colon [0,T]\rightarrow [0,\infty)$. We now show that~\eqref{eq: weak form of the macro eqn} implies~\eqref{eq: variational form of the macro eqn}. Since $\bar{H}$ is convex, we have
\begin{align}
\label{eq: usage convexity of mac H}
\langle g-PA^{-1}NJP^t g,\nabla_Y\bar{H}(\eta)\rangle_Y&\leq (\bar{H}(\eta+g)-\bar{H}(\eta))+(\bar{H}(\eta- PA^{-1}NJP^t g)-\bar{H}(\eta))\notag
\\&=-2\bar{H}(\eta)+\bar{H}(\eta+g)+\bar{H}(\eta- PA^{-1}NJP^t g).
\end{align}
Substituting~\eqref{eq: usage convexity of mac H} into~\eqref{eq: weak form of the macro eqn}, we obtain~\eqref{eq: variational form of the macro eqn}
\begin{align}
&\int_0^T\langle g,(\bar{A})^{-1}\eta\rangle_Y\dot{\beta}(t)dt\notag
\\&\leq -2\int_0^T\bar{H}(\eta)\beta(t)\,dt+\int_0^T\left[\bar{H}(\eta+g)+\bar{H}(\eta-PA^{-1}NJP^t g)\right]\beta(t)dt.
\end{align}
Next we show~\eqref{eq: variational form of the macro eqn} implies~\eqref{eq: weak form of the macro eqn}. Take $\tilde{g}=\varepsilon g$ in~\eqref{eq: variational form of the macro eqn}, for some $\varepsilon>0$ and $g\in Y$, we get
\[
\int_0^T\langle g,(\bar{A})^{-1}\eta\rangle_Y\dot{\beta}(t)dt\leq\int_0^T\left[\frac{\bar{H}(\eta+\varepsilon g)-\bar{H}(\eta)}{\varepsilon}+\frac{\bar{H}(\eta-\varepsilon PA^{-1}NJP^t g)-\bar{H}(\eta)}{\varepsilon}\right]\beta(t)\, dt.
\]
By passing to the limit~$\varepsilon\rightarrow 0$, we get
\[
\int_0^T\langle g,(\bar{A})^{-1}\eta\rangle_Y\dot{\beta}(t)dt\leq\int_0^T\langle g-PA^{-1}NJP^t g,\nabla_Y\bar{H}(\eta)\rangle_Y\beta(t)\,dt.
\]
Similarly now by taking $\tilde{g}=-\varepsilon g$, we obtain the opposite inequality
\[
\int_0^T\langle g,(\bar{A})^{-1}\eta\rangle_Y\dot{\beta}(t)dt\geq\int_0^T\langle g-PA^{-1}NJP^t g,\nabla_Y\bar{H}(\eta)\rangle_Y\beta(t)\,dt.
\]
Thus~\eqref{eq: weak form of the macro eqn} is proven.
\end{proof}

\begin{lemma}
\label{lem: compactness}
Let $\{\eta^{\ell}\}_{l=1}^\infty$ be a sequence of solutions of~\eqref{eq: macro eqn} with initial data $\eta^{\ell}_0$ satisfying $\|\bar{\eta}_0^{\ell}\|_{L^2}\leq C$. There exists a constant $C$ independent of $l$ such that
\begin{align} 
& \int_0^T \left\langle\frac{d\eta^{\ell}}{dt}(t), (\bar{A})^{-1}\frac{d\eta^{\ell}}{dt}(t)\right\rangle\,dt\leq C,\label{eq: integrate H-1 norm} \\
& \sup_{t\in [0,T]}\langle \eta^{\ell}(t),\eta^{\ell}(t)\rangle_Y\leq C\label{eq: eta norm}.
\end{align} 
As a consequence, there is a subsequence of the sequence of the associated step functions $\bar{\eta}^{\ell}$ and a function $\eta_*$ such that
\[
\bar{\eta}^{\ell}\rightharpoonup \eta_*\quad\text{weak-* in}\quad L^\infty(L^2)=(L^1(L^2))^*. 
\] 
\end{lemma}
\begin{proof}
According to proof of~\eqref{eq: bound of macro Fisher term}, we have 
\begin{equation}
\label{eq: uper bound of H}
\bar{H}(\eta^{\ell}(t))\leq e^{C(T+1)} \bar{H}(\eta^{\ell}_0)\quad\text{for all}~~ t\in[0,T]. 
\end{equation}

Since $\bar{H}$ is uniformly convex, we obtain
\[
\langle \eta^{\ell}(t),\eta^{\ell}(t)\rangle_Y\leq C(\bar{H}(\eta^{\ell}(t))+1)\leq Ce^{CT} (\bar{H}(\eta^{\ell}_0) + 1)\leq C,
\]
which is~\eqref{eq: eta norm}. Now we establish~\eqref{eq: integrate H-1 norm}.  From~\eqref{eq: macro eqn}, we have
\begin{align*}
&\langle\dot{\eta}^{\ell}(t), (\bar{A})^{-1}\dot{\eta}^{\ell}(t)\rangle=\langle\bar{A}(I+PJA^{-1}NP^t)\nabla_Y\bar{H}(\eta^{\ell}(t)),(I+PJA^{-1}NP^t)\nabla_Y\bar{H}(\eta^{\ell}(t))\rangle\\
&\qquad\leq 2(\langle\bar{A}\nabla_Y\bar{H}(\eta^{\ell}(t)),\nabla_Y\bar{H}(\eta^{\ell}(t))\rangle+\langle\bar{A}PJA^{-1}NP^t\nabla_Y\bar{H}(\eta^{\ell}(t)),PJA^{-1}NP^t\nabla_Y\bar{H}(\eta^{\ell}(t))\rangle)
\\&\qquad\overset{\eqref{eq: ineq2}}{\leq}2(\langle\bar{A}\nabla_Y\bar{H}(\eta^{\ell}(t)),\nabla_Y\bar{H}(\eta^{\ell}(t))\rangle+c|\nabla_Y\bar{H}(\eta^{\ell}(t))|^2)
\\&\qquad\overset{(iii)}{\leq} 2(\langle\bar{A}\nabla_Y\bar{H}(\eta^{\ell}(t)),\nabla_Y\bar{H}(\eta^{\ell}(t))\rangle+C(\bar{H}(\eta^{\ell}(t))+1))
\end{align*}
Therefore,
\begin{align*}
\int_0^T\langle\dot{\eta}^{\ell}(t), (\bar{A})^{-1}\dot{\eta}^{\ell}(t)\rangle\,dt&\leq 2\int_0^T(\langle\bar{A}\nabla_Y\bar{H}(\eta^{\ell}(t)),\nabla_Y\bar{H}(\eta^{\ell}(t))\rangle+C(\bar{H}(\eta^{\ell}(t))+1))\,dt
\\&\overset{\eqref{eq: bound of macro Fisher term},\eqref{eq: uper bound of H}}{\leq}C,
\end{align*}
which is~\eqref{eq: integrate H-1 norm}.
\end{proof}

\begin{lemma}
\label{lem: estimates limit}
Let $\{\eta^{\ell}\}_1^\infty$ be a sequence of solutions of~\eqref{eq: macro eqn} satisfying Lemma~\ref{lem: compactness}. We take any subsequence such that the associated step functions weak-* converge in $(L^1(L^2))^*$ to a limit $\eta_*$. Then on any bounded time interval, we have
\begin{equation}
\label{eq: regularity of the limt}
\eta_*\in L^\infty(L^2),\quad\frac{\partial \eta_*}{\partial t}\in L^2(H^{-1}),\quad \varphi'(\eta_*)\in L^2(L^2).
\end{equation}
\begin{proof}
Having the estimate in Lemma~\ref{lem: compactness}, the proof of this Lemma is the same as that of Lemma 35 in~\cite{GOVW09}; hence we omit it here.
\end{proof}
\end{lemma}

\begin{lemma}
If $\overline{g^{\ell}}\rightarrow \bar{g}$ strongly in $H^{-1}(\T)$  and $\sup_{\ell}\Vert \overline{g^\ell}\Vert_{L^2}<\infty$ then $-\overline{PA^{-1}JNP^tg^{\ell}}\rightarrow G$ strongly in $L^2(\T)$ where $G$ is the primitive of $\bar{g}$.
\end{lemma}

$\overline{PA^{-1}JNP^tg^{\ell}}$ is the step function associated to $PA^{-1}JNP^tg^{\ell}$, as in (\ref{def: bar x}). We only formally gave the definition for elements of $X_{N,0}$, while $PA^{-1}JNP^tg^{\ell}$ is in $Y_{M,0}$, but since $Y_{M,0} \sim X_{M,0}$, this is not an issue, and we just use the same definition, with mesh size $M^{-1}$ instead of $N^{-1}$.

\begin{proof}
Set
\begin{equation}
D=N\begin{pmatrix}
1&-1&&&\\
&1&-1&&\\
&&\ddots&\ddots&\\
&&&1&-1\\
-1&&&&1
\end{pmatrix},
\end{equation}
then we can write
\[
A=DD^T,\quad J=\frac{1}{2}(D^T-D).
\]
Hence
\begin{equation}
\label{eq: JA^-1}
JA^{-1}=\frac{1}{2}(D^T-D)(D^TD)^{-1}=\frac{1}{2}(D^T-D)D^{-1}(D^T)^{-1}=\frac{1}{2}(D^{-1}-(D^T)^{-1}).
\end{equation}
The inverse of $D$ and $D^T$ can be computed explicitly
\[
D^{-1}=\frac{1}{2N}
\begin{pmatrix}
 1\\
    & 1 & & \text{\huge1}\\
    & & 1\\
    & \text{\huge{-1}} & & 1\\
    & & & & 1
\end{pmatrix},\qquad (D^T)^{-1}=(D^{-1})^T.
\]
So we obtain
\begin{equation}
\label{eq: D^-1 -D^T^-1}
D^{-1}-(D^T)^{-1}=\frac{1}{N}
\begin{pmatrix}
 0\\
    & 0 & & \text{\huge1}\\
    & & 0\\
    & \text{\huge-1} & & 0\\
    & & & & 0
\end{pmatrix},
\end{equation}
Let 
$
\xi=\begin{pmatrix}
\xi_1\\
\vdots\\
\xi_M
\end{pmatrix}\in Y=\R^M$ be given. We now compute $\overline{PA^{-1}JNP^t\xi}$ explicitly. 

By definition of $P^t$, we have
\begin{equation}
\label{eq: NP^txi}
NP^t\xi=\begin{pmatrix}
\xi_1\\
\vdots\\
\xi_1\\
\xi_2\\
\vdots\\
\xi_2\\
\vdots\\
\xi_M\\
\vdots\\
\xi_M
\end{pmatrix}\in \R^N=\R^{KM}.
\end{equation}
From~\eqref{eq: JA^-1},~\eqref{eq: D^-1 -D^T^-1} and~\eqref{eq: NP^txi}, we have
\begin{align*}
A^{-1}JNP^t\xi=\frac{1}{2N}\begin{pmatrix}
K(\xi_1+\cdots+\xi_M)\\
(K-1)\xi_1+K(\xi_2+\cdots+\xi_M)\\
\vdots\\
K(\xi_2+\cdots+\xi_M)\\
(K-1)\xi_2+K(\xi_3+\cdots+\xi_M)\\
\vdots\\
K(\xi_3+\cdots+\xi_M)\\
\vdots\\
\xi_M
\end{pmatrix}-\frac{1}{2N}\begin{pmatrix}
\xi_1\\
2\xi_1\\
\vdots\\
K\xi_1\\
K \xi_1+\xi_2\\
\vdots\\
K(\xi_1+\cdots+\xi_M)
\end{pmatrix}.
\end{align*}
Therefore, by definition of $P$,
\begin{align*}
PA^{-1}JNP^t\xi&=\frac{1}{2M}\begin{pmatrix}
\xi_1+\xi_2+\cdots+\xi_M\\
\xi_2+\cdots+\xi_M\\
\vdots\\
\xi_M
\end{pmatrix}-\frac{1}{2M}\begin{pmatrix}
\xi_1\\
\xi_1+\xi_2\\
\vdots\\
\xi_1+\xi_2+\cdots+\xi_M\\
\end{pmatrix}.
\end{align*}

Recall that the $\xi_i$ sum up to $0$, and hence we can re-write the above equality as
\begin{align*}
-PA^{-1}JNP^t\xi&=\frac{1}{2M}\begin{pmatrix}
\xi_1\\
\xi_1+\xi_2\\
\vdots\\
\xi_1+\xi_2+\cdots+\xi_M\\
\end{pmatrix}+\frac{1}{2M}\begin{pmatrix}
0\\
\xi_1\\
\vdots\\
\xi_1+\xi_2+\cdots+\xi_{M-1}
\end{pmatrix}
\\&=\frac{1}{M}\begin{pmatrix}
\xi_1\\
\xi_1+\xi_2\\
\vdots\\
\xi_1+\xi_2+\cdots+\xi_M\\
\end{pmatrix}-\frac{1}{2M}\begin{pmatrix}
\xi_1\\
\xi_2\\
\vdots\\
\xi_M\\
\end{pmatrix}.
\end{align*}
It follows that
\begin{equation*}
-\overline{PA^{-1}JNP^t\xi}=\Upsilon_\xi-\frac{1}{2M}\overline{\xi},
\end{equation*}
where $\Upsilon_\xi$ denotes the primitive of $\overline{\xi}$. By the assumption $\overline{g^{\ell}}\rightarrow \overline{g}$ strongly in $H^{-1}$ and by definition of $H^{-1}$-norm, we have
\begin{equation*}
\Upsilon_{g^{\ell}}=-\overline{PA^{-1}JNP^t g^{\ell}}+\frac{1}{2M}\overline{ g^{\ell}}\rightarrow \Upsilon_{g}\equiv G~~\text{strongly in}~ L^2(\T).
\end{equation*}
The assertion then follows because
\begin{equation*}
\Vert -\overline{PA^{-1}JNP^t g^{\ell}} -G\Vert_{L^2}\leq \Vert -\overline{PA^{-1}JNP^t g^{\ell}}+\frac{1}{2M}\overline{ g^{\ell}}-G\Vert_{L^2} +\frac{1}{2M}\|\overline{ g^{\ell}}\|_{L^2}\rightarrow 0.
\end{equation*}

\end{proof}

\begin{lemma}
\label{lem: convergence weak form}
Suppose that the sequence $\eta^{\ell}$ satisfies (\ref{eq: integrate H-1 norm}), (\ref{eq: eta norm}) and (\ref{eq: variational form of the macro eqn}), and consider a subsequence such that 
\[
\bar{\eta}^{\ell}\rightharpoonup \eta_*\quad\text{weak-* in}\quad L^\infty(L^2)=(L^1(L^2))^*. 
\] 
holds. Let $\xi^{\ell} = \pi_{\ell}(\xi + \eta^{\ell}) - \eta^{\ell}$, where $\xi$ is an arbitrary $L^2$ function and $\pi_{\ell}$ is the $L^2$-projection onto elements of $Y$. Let $\Xi$ be the primitive with average $0$ of $\xi$. Then we have

(i) $$\underset{\ell}{\liminf} \hspace{1mm} \int_0^T{\bar{H}(\eta^{\ell}(t))\beta(t)dt} \geq \int_0^T{\int_{\T}{\varphi(\eta_*(t,\theta))\beta(t)d\theta} dt};$$

(ii) $$\underset{\ell}{\lim} \hspace{1mm} \int_0^T{\bar{H}(\eta^{\ell}(t) + \xi^{\ell}(t))\beta(t)dt} = \int_0^T{\int_{\T}{\varphi(\eta_*(t,\theta) + \xi(\theta))\beta(t)d\theta}dt};$$

(iii) $$ \underset{\ell}{\lim} \hspace{1mm} \int_0^T{\bar{H}(\eta^{\ell}(t) - PA^{-1}JNP^t\xi^{\ell}(t))\beta(t)dt} = \int_0^T{\int_{\T}{\varphi(\eta_*(t,\theta) - \Xi(\theta))\beta(t)d\theta}dt};$$

(iv) $$\underset{\ell}{\lim} \hspace{1mm}  \int_0^T{\langle \xi^{\ell}(t), \bar{A}^{-1}\eta^{\ell}(t) \rangle_Y \dot{\beta}(t)dt} = \int_0^T{\langle \xi(\theta), \eta_*(t, \theta)  \rangle_{H^{-1}} \dot{\beta}(t)dt}. $$

\end{lemma}

\begin{proof} (i), (ii) and (iv) have already been proven in Lemma 37 of \cite{GOVW09}, so we only have to prove (iii).

Recall that $\bar{H}(y) = M^{-1}\sum \psi_K(y_i) + N^{-1}\log \bar{Z}$, where $\psi_K$ was defined in (\ref{def: psi K}). Since $\bar{\eta}^{\ell}$ converges to $\eta_*$ and $\overline{PA^{-1}JNP^t\xi^{\ell}(t)}$ converges to $\Xi$, by weak lower-semi continuity and the uniform convergence of $\psi_K$ to $\varphi$ (see Proposition 31 in \cite{GOVW09}) we immediately get 
$$ \underset{\ell}{\liminf} \hspace{1mm}  \int_0^T{\bar{H}(\eta^{\ell}(t) - PA^{-1}JNP^t\xi^{\ell}(t))\beta(t)dt} \geq \int_0^T{\int_{\T}{\varphi(\eta_*(t,\theta) - \Xi(\theta))\beta(t)d\theta}dt}$$
so we only need to prove the associated upper bound. Let $g^{\ell}(t)$ be a sequence of elements of $Y$ such that $\bar{g}^{\ell}$ strongly converges in $L^{\infty}(L^2)$ to $\eta_* - \Xi$. Since we then have 
$$\int_0^T{\bar{H}(g^{\ell}(t))\beta(t)dt} \longrightarrow \int_0^T{\int_{\T}{\varphi(\eta_*(t,\theta) - \Xi(\theta))\beta(t)d\theta}dt}$$
we only need to show that 
$$\underset{\ell}{\limsup} \hspace{1mm}  \int_0^T{\bar{H}(\eta^{\ell}(t) - PA^{-1}JNP^t\xi^{\ell}(t))\beta(t)dt} - \int_0^T{\bar{H}(g^{\ell}(t))\beta(t)dt} \leq 0.$$
Let $A_M$ be the discrete Laplacian with scaling factor $M^2$ on $Y$. Since $\psi_K$ is convex, we have

\begin{align}
\bar{H}(\eta^{\ell}(t) &- PA^{-1}JNP^t\xi^{\ell}(t)) - \bar{H}(g^{\ell}(t)) = \frac{1}{M}\underset{i = 1}{\stackrel{M}{\sum}} \hspace{1mm} \psi_K(\eta^{\ell}_i -( PA^{-1}JNP^t\xi^{\ell})_i) - \psi_K(g^{\ell}_i) \notag \\
&\leq  \frac{1}{M}\underset{i = 1}{\stackrel{M}{\sum}} \hspace{1mm} \psi_K'(\eta^{\ell}_i -( PA^{-1}JNP^t\xi^{\ell})_i)(\eta^{\ell}_i -( PA^{-1}JNP^t\xi^{\ell})_i - g^{\ell}_i) \notag \\
&= \langle \nabla \bar{H}(\eta^{\ell} - PA^{-1}JNP^t\xi^{\ell}),( \eta^{\ell} - PA^{-1}JNP^t\xi^{\ell} - g^{\ell}) \rangle_Y \notag \\
&\leq \langle A_M \nabla \bar{H}(\eta^{\ell} - PA^{-1}NP^t\xi^{\ell}), \nabla \bar{H}(\eta^{\ell} - PA^{-1}NP^t\xi^{\ell}) \rangle_Y^{1/2} \notag \\
& \hspace{5mm} \times  \langle A_M^{-1}( \eta^{\ell} - PA^{-1}NP^t\xi^{\ell} - g^{\ell}) , ( \eta^{\ell} - PA^{-1}NP^t\xi^{\ell} - g^{\ell})  \rangle_Y^{1/2} \notag 
\end{align}

Since $\langle A_M^{-1} \cdot, \cdot \rangle_Y$ behaves like the squared $H^{-1}$ norm, the fact that $ \eta^{\ell} - PA^{-1}NP^t\xi^{\ell}$ and $ g^{\ell}$ converge to the same limit in $L^{\infty}(H^{-1})$ implies that 
$$\langle A_M^{-1}( \eta^{\ell} - PA^{-1}NP^t\xi^{\ell} - g^{\ell}) , ( \eta^{\ell} - PA^{-1}NP^t\xi^{\ell} - g^{\ell})  \rangle_Y \longrightarrow 0,$$
and therefore it will be enough to show that 
$$\int_0^T{ \langle A_M \nabla \bar{H}(\eta^{\ell} - PA^{-1}NP^t\xi^{\ell}), \nabla \bar{H}(\eta^{\ell} - PA^{-1}NP^t\xi^{\ell}) \rangle_Y dt} < C.$$

Since under our assumptions $\psi_K'$ is bi-Lipschitz, we have
\begin{align}
& \langle A_M \nabla \bar{H}(\eta^{\ell} - PA^{-1}NP^t\xi^{\ell}), \nabla \bar{H}(\eta^{\ell} - PA^{-1}NP^t\xi^{\ell}) \rangle_Y \notag \\
& \hspace{3mm} = M\underset{i = 1}{\stackrel{M}{\sum}} \hspace{1mm}( \psi_K'(\eta^{\ell}_{i+1} -  ( PA^{-1}NP^t\xi^{\ell})_{i+1}) -  \psi_K'(\eta^{\ell}_{i} -  ( PA^{-1}NP^t\xi^{\ell})_{i}))^2 \notag \\
& \hspace{3mm} \leq CM\underset{i = 1}{\stackrel{M}{\sum}} \hspace{1mm}( \eta^{\ell}_{i+1} -  ( PA^{-1}NP^t\xi^{\ell})_{i+1} -  (\eta^{\ell}_{i} -  ( PA^{-1}NP^t\xi^{\ell})_{i}))^2 \notag \\
& \hspace{3mm} \leq CM\underset{i = 1}{\stackrel{M}{\sum}} \hspace{1mm} ( \eta^{\ell}_{i+1} -  \eta^{\ell}_i)^2 + ( ( PA^{-1}NP^t\xi^{\ell})_{i+1} -  ( PA^{-1}NP^t\xi^{\ell})_i)^2 \notag \\
& \hspace{3mm} \leq CM\underset{i = 1}{\stackrel{M}{\sum}} \hspace{1mm} (\psi_K'( \eta^{\ell}_{i+1}) -  \psi_K'(\eta^{\ell}_i))^2 + \frac{C}{M}\underset{i = 1}{\stackrel{M}{\sum}} \hspace{1mm} (\xi^{\ell}_{i+1} - \xi^{\ell}_i)^2 \notag \\
& \hspace{3mm} \leq C\langle A_M\nabla \bar{H}(\eta^{\ell}), \nabla \bar{H}(\eta^{\ell}) \rangle + C||\bar{\xi}^{\ell}||_{L^2}^2. \notag
\end{align}
Since $\bar{\xi}^{\ell}$ converges in $L^2$, $||\bar{\xi}^{\ell}||_{L^2}^2$ is bounded. To conclude, we then only require (\ref{eq: integrate H-1 norm}) and the fact that 
\begin{equation}
\langle A_M y, y \rangle \leq C\langle \bar{A}y, y \rangle \hspace{5mm} \forall y \in Y.
\end{equation}
This statement is equivalent to bounding from below $A_M^{-1}$ by $\bar{A}^{-1}$. This does hold, since we have
\begin{align}
\langle \bar{A}^{-1}y, y \rangle_Y &= \frac{1}{N}\langle A^{-1}NP^ty, NP^ty \rangle_X \notag \\
&\leq C||\bar{NP^ty}||_{H^{-1}}^2 \notag \\
&\leq C||\bar{y}||_{H^{-1}}^2 \notag \\ 
&\leq C\langle A_M^{-1}y, y \rangle_Y \notag 
\end{align}
which concludes the proof.
\end{proof}

Finally, we need to prove uniqueness of solutions to the limiting PDE : 

\begin{lemma}
\label{lem: uniqueness}
Given an initial condition $\zeta_0$, there is at most one solution to (\ref{eq: limiting eqn}).
\end{lemma}

\begin{proof}
Let $\zeta_1$ and $\zeta_2$ be two solutions of (\ref{eq: limiting eqn}) with same initial condition. Let $F(t) := 2^{-1}||\zeta_1(t,\cdot) - \zeta_1(t,\cdot) ||_{H^{-1}}^2$, and let let $g_1$ and $g_2$ be mean-zero primitives (in space) of $\zeta_1$ and $\zeta_2$. Then, for any $\lambda > 0$, 
\begin{align}
F'(t) &= -\int_{\T}{(\varphi'(\zeta_1) - \varphi'(\zeta_2))(\zeta_1 - \zeta_2)d\theta} + \int_{\T}{(\varphi'(\zeta_1) - \varphi'(\zeta_2))(g_1 - g_2)d\theta} \notag \\
&\leq -\frac{\inf \varphi''}{2}\int_{\T}{(\zeta_1 - \zeta_2)^2d\theta} + \frac{\lambda}{2}\int_{\T}{(\varphi'(\zeta_1) - \varphi'(\zeta_2))^2d\theta} \notag \\
& \hspace{5mm} + \frac{1}{2\lambda}\int{(g_1 - g_2)^2d\theta} \notag \\
&\leq -\frac{\inf \varphi''}{2}\int_{\T}{(\zeta_1 - \zeta_2)^2d\theta} + \frac{\lambda \sup \varphi''}{2}\int_{\T}{(\zeta_1 - \zeta_2)^2d\theta} + \frac{1}{\lambda}F(t) \notag
\end{align}
Taking $\lambda = \frac{\inf \varphi''}{\sup \varphi''}$, we obtain a differential inequality which, by Gronwall's lemma, implies that $\zeta_1 = \zeta_2$.
\end{proof}

We can now prove Proposition \ref{prop: convergence eta}: 

\begin{proof}[Proof of  Proposition \ref{prop: convergence eta}]
According to Lemma \ref{lem: compactness}, we can consider a subsequence such that 
\[
\bar{\eta}^{\ell}\rightharpoonup \eta_*\quad\text{weak-* in}\quad L^\infty(L^2)=(L^1(L^2))^*. 
\] 
and strongly in $L^{\infty}(H^{-1})$. By Lemma \ref{lem: estimates limit}, $\eta_*$ satisfies (\ref{eq: regularity of weak sol}). 
According to Lemma \ref{lem: variational}, $\eta^{\ell}$ satisfies (\ref{eq: variational form of the macro eqn}). Passing to the limit using Lemma \ref{lem: convergence weak form}, we see that $\eta_*$ satisfies (\ref{eq: variational form of the limiting eqn}), and therefore is a weak solution of~\eqref{eq: limiting eqn}. 

Since Lemma \ref{lem: uniqueness} guarantees uniqueness of the weak solution, the full sequence $(\eta^{\ell})_{\ell}$ converges to the unique weak solution of~\eqref{eq: limiting eqn}. 

\end{proof}

\end{document}